\documentclass[12pt,reqno]{amsart} 
\usepackage[utf8]{inputenc}
\usepackage{amsthm,amsmath,amsfonts,amssymb}

\usepackage{graphicx}
\usepackage{color}
\usepackage[dvipsnames]{xcolor}
\usepackage{subcaption}
\usepackage{hyperref}
\usepackage{url}
\usepackage{mathtools}
\usepackage{enumerate}
\usepackage{array}
\usepackage{paralist}

\newtheorem{theorem}{Theorem}[section]
\newtheorem{corollary}[theorem]{Corollary}

\newtheorem{lemma}[theorem]{Lemma}

\theoremstyle{remark}
\newtheorem{remark}[theorem]{Remark}
 
\numberwithin{equation}{section}

\def\R{\ensuremath{I\!\!R}}
\def\Rp{\ensuremath{\R_{> 0}}}

\DeclareMathOperator\diag{diag}
\DeclareMathOperator\sign{sign}
\DeclareMathOperator\rank{rank}

\def\cal#1{\mathcal{#1}}
\def\Vh{\hat V}
\def\Om{\Omega}
\newcommand{\C}{\mathbb{C}}

\newcommand{\N}{{\mathbb N}}
\newcommand{\reff}[1]{(\ref{#1})}
\newcommand{\bpm}{\begin{pmatrix}}\newcommand{\epm}{\end{pmatrix}}
\newcommand{\spr}[1]{\langle #1 \rangle}

\def\oc{\overline{c}}

\def\wt{\widetilde}
\def\tew{\textwidth}
\def\pa{{\partial}}
\def\pdep{{\tt pde2path}}
\def\eps{\varepsilon}
\def\del{\delta}
\def\ddt{\frac{\rm d}{{\rm d}t}}
\def\dd{\, {\rm d}}
\def\CO{{\cal O}}
\def\vt{\vartheta}
\def\mlab{{\tt Matlab}}
\def\kap{\kappa}
\def\er{{\rm e}}
\def\ri{{\rm i}}
\def\re{{\rm Re}}\def\im{{\rm Im}}

\allowdisplaybreaks

\begin{document} 
\title[Monomial steady state parameterizations and spatial instabilities]
{Conditions for spatial instabilities and pattern formation from 
monomial steady state parameterizations}

\author{Carsten Conradi}
\address{Hochschule für  Technik und Wirtschaft\\ Berlin, Germany}
\email{carsten.conradi@htw-berlin.de}

\author{Maya Mincheva}
\address{Department of Mathematical Sciences, Northern Illinois University\\ DeKalb, IL, USA}
\email{mmincheva@niu.edu}

\author{Hannes Uecker}
\address{Institut für Mathematik, Carl von Ossietzky Universität Oldenburg\\ Germany} 
\email{hannes.uecker@uni-oldenburg.de}

\date{\today}

\begin{abstract}
  We study the onset of spatial instabilities in reaction networks
  where the spatially homogeneous system admits a steady state
  parameterization. We formulate a sufficient condition -- based on the 
  signs of the constant and leading coefficients of the characteristic
  polynomial of the linearized Jacobian scaled by the diffusion
  coefficients -- that guarantees a Turing-like instability to spatially
  inhomogeneous solutions on appropriately chosen domains $\Omega$. We
  also present a specific condition on the domain size $|\Omega|$
  required to trigger this instability. As a consequence of employing
  a monomial parameterization, these conditions take the form of
  algebraic polynomial inequalities involving only rate constants and
  diffusion coefficients. We apply these ideas to a network describing
  the sequential and distributive (de-)phosphorylation of a protein at
  two binding sites, ultimately deriving a condition involving only
  the four catalytic constants of the enzymes and the diffusion
  coefficients of the four enzyme-substrate complexes that guarantees
  a Turing-like instability.


\end{abstract}

\maketitle


\section{Introduction}

The emergence of
spatial patterns in biochemical networks is an active area of
research, see e.g.\ \cite{Krishnan2020,Menon2021}.
A seminal mechanism for pattern formation in reaction--diffusion systems 
has been introduced by Alan Turing in \cite{Turing1952} 
(c.f.\ for example, \cite{diff-010} and the
contributions \cite{diff-012,Krause2021} on the occasion of the 
70th anniversary of the paper). 
The mechanism consists of a reaction
system of two interacting chemicals. Two scenarios are considered: (i) the
homogeneous, well mixed case when diffusion is negligible and the
dynamics of the species can be described by a pair of ODEs; and
(ii), the inhomogeneous case when diffusion cannot be neglected and
the dynamics is described by a pair of reaction-diffusion PDEs. In
scenario (i) the steady state of this ODE system is 'stable' (i.e.\
all eigenvalues of its linearization have a negative real part). In
scenario (ii) this steady state corresponds to a homogeneous,
time-constant solution (a homogeneous steady state) of the PDE. In
Turing's example the diffusion coefficients are chosen in such a way
that this homogeneous steady state of the PDE is unstable (i.e.\ at
least one eigenvalue of its linearization is positive) and a
non-homogeneous time-constant solution emerges nearby. In this
setting the non-homogeneous steady state represents a pattern that
emerges when the homogeneous steady state is de-stabilized by
diffusion.

Given a reaction system with fixed rate constants yielding
a stable steady state of the ODE system, this 'de-stabilizing' only 
happens for selected values of the diffusion coefficients. Finding
such values for the diffusion coefficients is straightforward in the
two species case, see 
e.g. \cite{jost2013partial,jostmathematical,murray2001mathematicalII},
but can be challenging if the number of species increases. For an
arbitrary number of species only a limited number of results
exists, for example,~\cite{Satnoianu2005,Satnoianu2000}, and,
\cite{VSC23,VillarSepulveda2025}. In \cite{Satnoianu2000} a
necessary and a sufficient condition for a Turing Instability are
given in terms of the signs of the principal minors of the Jacobian
of the ODE system. Every principal minor corresponds to what the
authors call a subsystem and if the minor is of a certain sign, then
the subsystem is called \lq unstable\rq. The approach presented
in~\cite{VSC23} further extends the idea of unstable (sub-)systems:
given a subsystem of size $p$, the authors show that by suitably
choosing the diffusion coefficients and the domain one can force
$p$ eigenvalues of the linearization to approach the eigenvalues of
the subsystem while the remaining eigenvalues of the linearization
approach $-\infty$. And \cite{VillarSepulveda2025} allows for
cross-diffusion and exploits these ideas to designing diffusion
coefficients that enable instabilities.

Here we consider reaction networks that admit a monomial
parametrization with any number of species, their ODEs and the
corresponding reaction-diffusion PDEs on spatial domains of
dimension 1, 2 and 3. Networks that admit a monomial
parametrization are a generalization of the networks described in
\cite{PerezMillan2012} and have been introduced in
\cite{alg-043}. They form a special class of reaction networks, where
all positive steady states can be  parameterized. We first focus on
the existence of at least one positive real eigenvalue of the
linearization of the reaction-diffusion PDE. To this end we
analyze a polynomial representation of the determinant of the
linearization and formulate inequality conditions on the sign of its
leading and constant coefficient (inequalities~(\ref{eq:condi_a0})
\& (\ref{eq:condi_as})).  Given a monomial parametrization of all
positive steady states, these inequalities represent polynomial
conditions that are sufficient for the linearization to have at
least one positive eigenvalue. 
Based on this inequality condition we formulate conditions on the
domain (in 1-, 2- and 3-D) that are  sufficient for the existence of
such positive eigenvalues~(Theorem~\ref{theo:existence_mu_k}). If the
rate constants have been chosen in such a way that the steady state of
the ODE is stable, then an instability occurs
(Corollary~\ref{coro:cond-turing-inst}). We call the corresponding
instability \emph{Turing-like} for the following reason: 
Usually the term \emph{Turing instability} refers to
instabilities with respect to patterns with wavelengths independent
of the (asymptotically large) domain size. We, however, present
conditions guaranteeing that instabilities will always arise for the
lowest eigenvalues of the scalar Neumann Laplacian on $\Om$. Due to
this reduction we do not 
consider all types of instabilities, but {\em long wave} (aka {\em
  sideband}) types: the first instability is always with respect to
nonhomogeneous solutions of maximal wavelength allowed by the domain
(cf.~Remark~\ref{numrem} and \S\ref{numsec} for illustration by dispersion 
relations). 

Our approach is closely related to that of \cite{Satnoianu2000}: both 
are based on the fact that the determinant of a matrix equals the
product of all its eigenvalues and that consequently a change of the
sign of the determinant   indicates an odd number of real eigenvalues
crossing the imaginary axis. The authors of~\cite{Satnoianu2000}
exploit the fact that the determinant of the linearization can be
expressed in terms of sums of products of principal minors of the
Jacobian of the ODE system (weighted by products of diffusion
coefficients) and that stable and unstable subsystems lead to terms
with opposite signs. Hence the existence of at least one unstable
subsystem allows the determinant to change sign.

We, on the other hand, analytically compute the determinant of
the linearization evaluated at all homogeneous solutions
(with the help of the monomial parametrization). This
yields a polynomial that has coefficients that are rational in the
rate constants and diffusion coefficients.  Our inequalities 
(\ref{eq:condi_a0}) \&  (\ref{eq:condi_as}) then
guarantee that this polynomial changes sign (at least once).

As an application we consider the network from \cite{maya-bistab}. It
describes the double (de)phosphorylation of a protein $S$ by a kinase
$K$ and a phosphatase $F$ according to the following mechanism ($S_0$,
$S_1$ and $S_2$ refer to un-, mono-, and double-phosphorylated
substrate):
\begin{equation}
  \label{mapk}
  \begin{split}
    S_0 + K   &\xrightleftharpoons[k_{2}]{k_{1}} S_0K \xrightarrow{k_3}
     S_1 +K\xrightleftharpoons[k_{5}]{k_{4}} S_1 K\xrightarrow{k_6}
     S_2 +K,\\ 
     S_2+ F &\xrightleftharpoons[k_{8}]{k_{7}} S_2F
     \xrightarrow{k_9} S_1 +F\xrightleftharpoons[k_{11}]{k_{10}} S_1
     F\xrightarrow{k_{12}} S_0 +F\ .
  \end{split}
\end{equation}
Here $k_1$, \ldots, $k_{12}$ denote the reaction rate constants. There
are nine species, cf.\,(\ref{eq:variables_dd}) and three conserved 
quantities: the total amount of kinase, phosphatase and substrate.
A reaction-diffusion PDE derived from this network additionally requires 
nine diffusion coefficients. Our analysis shows that the catalytic
constants of the kinase with substrates $S_0$ ($k_3$), and $S_1$
($k_6$) as well as those of the phosphatase with substrates $S_2$
($k_9$), and $S_1$ ($k_{12}$) together with the diffusion 
coefficients of the enzyme substrate complexes $S_0K$ ($d_3$), $S_1K$
($d_5$), $S_2F$ ($d_8$) and $S_1 F$ ($d_9$) enable spatial 
instabilities in the following sense: as a consequence of
Theorem~\ref{theo:existence_mu_k} and
Corollary~\ref{coro:cond-turing-inst} we show here that if the
catalytic constants and the diffusion coefficients $d_3$, $d_5$, $d_8$
and $d_9$ are such that
\begin{equation}
  \label{eq:kc_Km_d_ineq}
  \frac{k_3\, k_9}{k_6\,
    k_{12}} < \frac{d_3 d_8}{d_5 d_9},
\end{equation}
then the linearization of the reaction diffusion PDEs derived from 
network~(\ref{mapk}) has at least one positive real eigenvalue -- for 
an appropriately chosen spatial domain. 
And if the corresponding steady state of the ODEs derived from
network~(\ref{mapk}) is stable, then a Touring-like instability
occurs. To illustrate our results we choose parameter values
satisfying~(\ref{eq:kc_Km_d_ineq}) and proceed with a detailed
numerical analysis on a 1D- and a 2D-domain; the $k_i$- and
$d_i$-values are given in Section~\ref{sec:Turing_dd}. On both
domains we find several non-homogeneous solutions and examine their
local stability by numerical simulations. 

The problem of genuine Turing instabilities for \reff{mapk}
remains open, as our conditions will always lead to instabilities
for the lowest eigenvalues of the Neumann Laplacian. 
Moreover, our analysis is based on matrix determinants,
i.e., on the products  of eigenvalues, and hence ignores
instabilities due to complex conjugate pairs of eigenvalues
crossing the  imaginary axis.  Hence, we cannot detect pairs of
complex-conjugate eigenvalues crossing.  Interestingly, when
examining the system~(\ref{mapk}) numerically, we found Hopf
bifurcations nonetheless. This finding is interesting in its own
right, in particular as the existence of Hopf bifurcations in the
ODE model of the system~(\ref{mapk}) is still an open
question~\cite{Conradi2020,feliu2024networkreductionabsencehopf}.
Such Hopf instabilities and associated bifurcations of time--periodic
solutions  do play a role in certain parameter regimes for the
system~(\ref{mapk}) (c.f.\ \S\ref{numsec}, where we investigate
standing wave solutions arising in a Hopf-bifurcation). Thus, while
our analysis is helpful to identify some sufficient conditions for
instabilities, which are also useful as guidance for the numerics  in
\S\ref{numsec}, it is only a first step towards a more complete
picture of pattern formation in the system~(\ref{mapk}). 
  
The paper is organized as follows: in \S~\ref{sec:MA-networks} we
introduce the notation, the dynamical systems defined by mass action
networks (ODEs and reaction-diffusion PDEs) and mass action networks
that admit a monomial parametrization. In
\S\ref{sec:turing-instabilities} we derive conditions for the
occurrence of spatial instabilities, in \S\ref{numsec} we apply our
results to network~(\ref{mapk}), and in Section~\ref{dsec} we briefly
summarize our results.

\section{The mass action network}
\label{sec:MA-networks}

See  for example \cite{Feinberg} for an in-depth introduction to 
chemical reaction networks with mass action kinetics. We use $c$ to
denote concentrations of chemical species, $\Gamma$ to denote the
stoichiometric matrix defined by the network, and $r(k,c)$ to denote
the reaction rates. Assuming $n$ chemical species and $r$ reactions,
$\Gamma$ is a $n\times r $ matrix, $\Gamma\in\R^{n\times r}$.
The reaction rates are given by a function $r:\R^r\times\R^n \to \R^r$
that depends on the (positive) rate constant vector $k\in\R^r$ and the
concentrations. For mass action networks it is a vector valued
function with monomial entries.

For network~(\ref{mapk}), we use 
\begin{align}
  \notag
  c_1 &= [S_0], & c_2 &= [K], & c_3 &= [S_0\, K], \\
  \label{eq:variables_dd}
  c_4 &= [S_1] & c_5 &= [S_1\, K] & c_6 &= [S_2]
  \\
  \notag
  c_7 &= [F] & c_8 &= [S_2\, F] & c_9 &= [S_1\,
  F]
\end{align}
to denote the species concentrations, under the assumption of 
mass action kinetics, one obtains the
following reaction rates for the network~(\ref{mapk}), 
\begin{subequations}
  \begin{align*}
    r_1(k,c) &= k_1\, c_1\, c_2, & 
    r_2(k,c) &= k_{2} c_{3}, &
    r_3(k,c) &= k_{3} c_{3},& r_4(k,c) &= k_{4} c_{2} c_{4},\\ 
    r_5(k,c) &= k_{5} c_{5}, &
    r_6(k,c) &= k_{6} c_{5}, &
    r_7(k,c) &= k_{7} c_{6} c_{7}, &
    r_8(k,c) &= k_{8} c_{8}, \\
    r_9(k,c) &= k_{9} c_{8},& 
    r_{10}(k,c) &= k_{10} c_{4} c_{7}, &
    r_{11}(k,c) &= k_{11} c_{9}, &
    r_{12}(k,c) &= k_{12} c_{9}, 
  \end{align*}
\end{subequations}
and the vector of reaction rates
\begin{equation}
  \label{eq:def_r}
  r(k,c) = \left(r_1(k,c),\, \ldots, r_{12}(k,c)\right)^T\ .
\end{equation}
In the ordering of (\ref{eq:variables_dd}) one obtains for the
stoichiometric matrix:
\begin{equation}
  \label{eq:Gam_dd}
  \Gamma = \left[
    \begin{array}{rrrrrrrrrrrr}
      -1 & 1 & 0 & 0 & 0 & 0 & 0 & 0 & 0 & 0 & 0 & 1 \\
      -1 & 1 & 1 & -1 & 1 & 1 & 0 & 0 & 0 & 0 & 0 & 0 \\
      1 & -1 & -1 & 0 & 0 & 0 & 0 & 0 & 0 & 0 & 0 & 0 \\
      0 & 0 & 1 & -1 & 1 & 0 & 0 & 0 & 1 & -1 & 1 & 0 \\
      0 & 0 & 0 & 1 & -1 & -1 & 0 & 0 & 0 & 0 & 0 & 0 \\
      0 & 0 & 0 & 0 & 0 & 1 & -1 & 1 & 0 & 0 & 0 & 0 \\
      0 & 0 & 0 & 0 & 0 & 0 & -1 & 1 & 1 & -1 & 1 & 1 \\
      0 & 0 & 0 & 0 & 0 & 0 & 1 & -1 & -1 & 0 & 0 & 0 \\
      0 & 0 & 0 & 0 & 0 & 0 & 0 & 0 & 0 & 1 & -1 & -1
    \end{array}
  \right]. 
\end{equation}

\subsection{The ODEs and PDEs derived from the mass action network}
\label{sec:odes-pdes-MA-network}

Let~$t$ denote time. For network~(\ref{mapk}), the ODEs derived from
network~(\ref{mapk}) in the well mixed spatially homogeneous scenario
read
\begin{subequations}\label{mapk-sys}
  \begin{align}
    \label{eq:ma_ode_1}
    \dot c_1 &= -c_{1} (k_{1} c_{2}) + k_{2} c_{3} + k_{12} c_{9}, \\
    \dot c_2 &= -c_{2} (k_{1} c_{1} + k_{4} c_{4}) + k_{2} c_{3} +
    k_{3} c_{3} + (k_{5}+k_{6}) c_{5}, \\
       \label{eq:ma_ode_3}
    \dot c_3 &=  -c_{3} (k_{2}+k_{3}) + k_{1} c_{1} c_{2}, \\
       \label{eq:ma_ode_4}
    \dot c_4 &= -c_{4} (k_{4} c_{2}+k_{10} c_{7}) + k_{3} c_{3} +
    k_{5} c_{5} + k_{9} c_{8} + k_{11} c_{9}, \\
       \label{eq:ma_ode_5}
    \dot c_5 &= -c_{5} (k_{5}+k_{6}) +  k_{4} c_{2} c_{4}, \\
    \dot c_6 &= -c_6 (k_{7} c_{7}) + k_{6} c_{5} + k_{8} c_{8}, \\
    \dot c_7 &= -c_7 (k_{10} c_{4}+k_{7} c_{6}) + (k_{8}  + k_{9})
    c_{8} + (k_{11}+k_{12}) c_{9}, \\
       \label{eq:ma_ode_8}
    \dot c_8 &= -c_8 (k_{8}+k_{9}) + k_{7} c_{6} c_{7}, \\
    \label{eq:ma_ode_9}
    \dot c_9 &= -c_{9}(k_{11}+k_{12}) + k_{10} c_{4} c_{7}, 
  \end{align}
\end{subequations}
which together with initial conditions $0<c_0\in\R^9$ we abreviate as 
\begin{equation}
  \label{eq:crn_ode}
  \dot c = \Gamma r(k,c), \; c(0) = c_0 \text{ with } c_0>0\ .
\end{equation}
Since $\Gamma r(k,c)$ is locally Lipschitz (for every
positive $k$), by standard results from ODE theory there exists 
a time $T>0$ (possibly infinite) and a unique solution 
$ c\in C^1([0,T],\Rp^n)$. If the matrix $\Gamma$ does not have full
row rank, i.e.\ if $\rank(\Gamma)=s<n$, then there exists a matrix $Z$
of maximal rank $n-s$ such that $Z^{'}\, \Gamma \equiv 0$. In this
case, solutions $c(t)$ with initial value $c(0) = c_0$ satisfy
\begin{displaymath}
  Z^{'} c(t) = Z^{'}c_0, 
\end{displaymath}
and the constant quantities $Z^{'}c_0$ are often referred to as total
amounts, or total concentrations, total masses. 

For network~(\ref{mapk}) and $\Gamma$ as in (\ref{eq:Gam_dd}) one
finds $\rank(\Gamma)=6$ and hence there exist $9\times 3$-matrices~$Z$
such that $Z^{'}\, \Gamma \equiv 0$, for example the matrix
\begin{displaymath}
  Z^{'} = \left[
    \begin{array}{*9{c}}
      0 & 1 & 1 & 0 & 1 & 0 & 0 & 0 & 0 \\
      0 & 0 & 0 & 0 & 0 & 0 & 1 & 1 & 1 \\
      1 & 0 & 1 & 1 & 1 & 1 & 0 & 1 & 1
    \end{array}
  \right]\ .
\end{displaymath}

\begin{remark}
  If $\rank(\Gamma)=n-s$, then the Jacobian $J(c)$ of
  the function~$\Gamma r(k,x)$ of a mass action network has at least $s$ 
  eigenvalues zero. Given a vector $k$ and a steady state $c$ we say
  that $c$ is stable, if all eigenvalues that are not identically zero
  have negative real part.
\end{remark}

In the spatially inhomogeneous scenario, the variables $c$ are defined
by solutions of the partial differential equation (PDE) 
\begin{subequations}
  \label{eq:pde}
  \begin{align}
    &\pa_t c= \Gamma r(k,x) + D\, \Delta c, 
      \intertext{
      where~$D$ is a diagonal matrix containing the
      positive diffusion coefficients and $\Delta$ is the element-wise 
      (negative) Laplace
      operator, i.e.\ $\Delta c_i =(\pa^2_{x_1} +
      \dots + \pa_{x_d}^2)c_i$, and the PDE is posed on some domain 
      $\Om$ with  (piecewise) smooth boundary $\partial \Omega$ and 
      Neumann (aka no-flux) boundary conditions 
      }
    &\pa_\nu c|_{x\in\partial \Omega} = 0, 
      \intertext{where $\nu$ is the outward unit normal and $\pa_\nu
      c$ the directional derivative, and with initial conditions,}
    &c|_{t_0}=c_0:\Om\to \R_+^9. 
  \end{align}
\end{subequations}
Though weaker notions of solution are possible, here solutions 
 of \reff{eq:pde} are functions 
 \begin{displaymath}
  c : [0,T] \times \Omega \to \Rp^n, (t,z) \mapsto (c_1(t,z), \ldots,
  c_n(t,z))^{'}\ , 
\end{displaymath}
at least once differentiable in $t$ and twice in $x$. Naturally, 
any solution $c:[0,T]\to\R_+^n$ of \reff{eq:crn_ode} satisfies the 
boundary conditions (\reff{eq:pde}b) and hence can be identified 
with a spatially homogeneous solution of \reff{eq:pde} to initial 
conditions $c_0(x)\equiv c_0$. 
We are interested in steady solutions $c^*$ of \reff{eq:crn_ode} and
\reff{eq:pde}. Such steady states are called stable if (in the ODE
case) for every $\eps>0$ there exists a $\delta>0$ such that
$\|c(0)-c^*\|_{\R^9}<\del$ implies that the solution $c(t)$ exist for
all $t>0$ and $\|c(t)-c^*\|_{\R^9}<\eps$ for all $t>0$. Otherwise
$c^*$ is called unstable. For the PDE the notion of stability in
principle depends on the chosen norms, but we can define
$c^*:\Om\to\R_+^9$ to be stable if
$\|c(0,\cdot)-c^*(\cdot)\|_\infty<\del$ implies that $c(t,\cdot)$
exists for all $t>0$ and $\|c(t,\cdot)-c^*(\cdot)\|_\infty<\eps$ 
for all $t>0$. As explained above, we are in particular interested 
in the situation where a steady state $c^*\in\R^9$ is stable in the 
ODE, but the associated function $c^*(\cdot)\equiv c^*$ is unstable 
in the PDE, i.e., unstable wrt spatially inhomogeneous perturbations. 

\subsection{Mass action networks that admit a monomial parametrization} 
\label{sec:mono-para}
If a mass action network admits a monomial parametrization, then
by~\cite[Lemma~3.5]{alg-043} there exists a set
$\mathcal{K}\subseteq\Rp^r$, a function $\psi : \mathcal{K} \to \R^n$
and a matrix $A\in\R^{p\times n}$ such that 
\begin{equation}
  \label{eq:equi_mono_para}
  \Gamma r(k,c) = 0 \Leftrightarrow k\in \mathcal{K} \text{ and } c = \psi(k)
  \circ \xi^A \text{ for some $0<\xi\in\R^p$ and $A\in\R^{n\times p}$}, 
\end{equation}
where $\circ$ denotes the element-wise product of two vectors, and 
$\displaystyle (\xi^A)_j=\prod_{i=1}^p \xi_i^{A_{ij}}$. 
One consequence of this equivalence is that for every $k\in\mathcal{K}$,
one obtains infinitely many positive steady states (for that particular~$k$)
parameterized by~$\xi$. Many signal transduction networks admit a
monomial parametrization, for example the networks studied in
\cite{Conradi2019,Conradi2015,Holstein2013,millan2015mapk,PerezMillan2012}
and some of the MESSI systems discussed in \cite{PerezMillan2018}. 
For an in-depth discussion see~\cite{alg-043}.

\subsubsection{A monomial parametrization for network~(\ref{mapk})}

By setting $(c_1,c_2,c_7)=(\xi_1,\xi_2,\xi_7)$ and setting
(\ref{eq:ma_ode_1}) -- (\ref{eq:ma_ode_9}) to zero and solving for
$c_3$, $c_4$, $c_5$, $c_6$, $c_8$, $c_9$ we obtain the following
parametrization of steady states for network~(\ref{mapk}):
\begin{subequations}
  \begin{equation}
    (\xi_1,\xi_2,\xi_3) \mapsto (c_1,\ldots, c_9)
  \end{equation}
  with
  \begin{align}
    \notag
    c_1 &= \xi_{1},\;
          c_2 = \xi_{2},\;
          c_3 =
          \frac{k_{1}}{k_{2}+k_{3}} \xi_{1} \xi_{2},\;
          c_4 = \frac{k_{1} k_{3} (k_{11}+k_{12})}{k_{10} k_{12}  
          (k_{2}+k_{3})} \frac{\xi_{1} \xi_{2}}{\xi_{3}}, \\
    \label{eq:mono_para}
    c_5&= \frac{k_{1} k_{3} k_{4} (k_{11}+k_{12})}{k_{10} k_{12}
         (k_{2}+k_{3}) (k_{5}+k_{6})} \frac{\xi_{1} \xi_{2}^2}{ \xi_{3}},\quad
         c_6 = \frac{k_{1} k_{3} k_{4} k_{6} (k_{11}+k_{12})
         (k_{8}+k_{9})}{k_{10} k_{12} k_{7} k_{9} (k_{2}+k_{3})
         (k_{5}+k_{6})} \frac{\xi_{1} \xi_{2}^2}{\xi_{3}^2}, \\
    \notag
    c_7 &= \xi_{3},\; 
          c_8 = \frac{k_{1} k_{3} k_{4} k_{6}  (k_{11}+k_{12})}{k_{10} k_{12}
          k_{9} (k_{2}+k_{3}) (k_{5}+k_{6})} \frac{\xi_{1}
          \xi_{2}^2}{\xi_{3}},\;
          c_9 = \frac{k_{1} k_{3} }{k_{12} (k_{2}+k_{3})}
          \xi_{1} \xi_{2}\ .
  \end{align}
\end{subequations}
In this case the monomial parametrization is given by the following
$\psi(k)$ and~$A$: 
\begin{displaymath}
  \begin{split}
    \psi(k)= \biggl(&
    1,1,\frac{k_{1}}{k_{2}+k_{3}}, \frac{k_{1} k_{3}
      (k_{11}+k_{12})}{k_{10} k_{12} (k_{2}+k_{3})},
    \frac{k_{1} k_{3} k_{4} (k_{11}+k_{12})}{k_{10} k_{12} (k_{2}+k_{3})
      (k_{5}+k_{6})}, \\
    & 
    \frac{k_{1} k_{3} k_{4} k_{6} (k_{11}+k_{12}) (k_{8}+k_{9})}{k_{10}
      k_{12} k_{7} k_{9} (k_{2}+k_{3}) (k_{5}+k_{6})}, 1,
    \frac{k_{1} k_{3} k_{4} k_{6}  (k_{11}+k_{12})}{k_{10} k_{12}
      k_{9} (k_{2}+k_{3}) (k_{5}+k_{6})},
    \frac{k_{1} k_{3} }{k_{12} (k_{2}+k_{3})}
    \biggr), 
  \end{split}
\end{displaymath}
and
\begin{displaymath}
  A = \left[
    \begin{array}{*9{r}}
      1 & 0 & 1 & 1 & 1 & 1 & 0 & 1 & 1 \\
      0 & 1 & 1 & 1 & 2 & 2 & 0 & 2 & 1 \\
      0 & 0 & 0 & -1 & -1 & -2 & 1 & -1 & 0
    \end{array}
  \right]\ .
\end{displaymath}

\section{Turing--like instabilities}
\label{sec:turing-instabilities}

As described above, Turing's approach to pattern formation consists of
two ingredients: a spatially homogeneous ODE of the
form~(\ref{eq:crn_ode}) and a reaction-diffusion PDE of the
form~(\ref{eq:pde}), where the ODE has a stable steady state. This
also defines a spatially homogeneous steady state solution of the PDE,
and if the diffusion coefficients are such that this homogeneous
steady state is unstable, then one speaks of Turing (--like)
instability.

Throughout this section we will assume that our system admits a
monomial parametrization of some positive steady states $\bar c$,
suppressing the dependency on the~$k_i$ and~$\xi_i$ for notational
simplicity. We will use the symbol~$J(\bar c)$ to denote the Jacobian
of $\Gamma r(k,c)$ evaluated at $\oc$, again suppressing the
$k_i$-dependency for simplicity. We assume that $\oc$ is (neutrally)
stable in the ODE, i.e., that all nonzero eigenvalues of $J(\oc)$ 
have negative real parts and aim to identify conditions which yield
that $\oc$ is unstable in the PDE. For that we use the MP to derive
the condition \reff{eq:kc_Km_d_ineq} after reducing the vector valued
linearized problem to the lowest eigenvalues of the scalar Neumann
Laplacian on $\Om$.

\subsection{Eigenvalues and determinants}
\label{sec:Turing_general} 

To identify Turing--like instabilities, the PDE~(\ref{eq:pde}) is
linearized at the spatially homogeneous steady state $\bar c$. 
Introducing $w(z,t) = c(z,t) - \bar c$,  the local (i.e., linearized) 
dynamics are described by the linear PDE
\begin{equation}
  \label{lpde}
  \dot w = J(\bar c) w + D \Delta w,
\end{equation}
with Neumann BCs for $w$, and initial condition $w(0,z)= w_0(z)$. 
The stability of $w\equiv 0$ and hence of the homogeneous steady state
$\bar c$ is then determined by the eigenvalues of the matrices
\begin{equation}
  \label{eq:def_Ak}
  A_{\ell} = J(\bar c) - \mu_{\ell} D,\; \ell = 0, 1, 2, \ldots
\end{equation}
where $\mu_{\ell}$ are the eigenvalues of the Neumann Laplacian~$-\Delta$
on the domain~$\Omega$. In a first step we neglect the fact that
$A_{\ell}$ is defined for the eigenvalues of $-\Delta$ and consider
the matrix function 
\begin{equation}
  \label{eq:def_A_mu}
  A(\mu) : \R \to \R^{n\times n},\; \mu \mapsto J(\bar c)- \mu D\ .
\end{equation}
In a second step we identify values~$\bar \mu$ for which the
determinant~$\det(A(\mu))$ has certain properties, and in a third step
we combine this with conditions on the domain $\Omega$ that guarantee
that there exist eigenvalues~$\mu_{\ell}$ of~$-\Delta$ on~$\Omega$
with $0<\mu_{\ell}<\bar\mu$.

\begin{lemma}
  \label{lemma:det_JD_not_zero}
  Let $A(\mu)$ be as in (\ref{eq:def_A_mu}) where $J(\bar c)$ and $D$
  are matrices of dimension $n\times n$ and satisfy $\rank(J(\bar
  c))=s<n$ and $D=\diag(d_1,\ldots,d_n)$ with $d_i>0$. 
  Then
  \begin{enumerate}[{(}i{)}]
  \item\label{item:A_factors}
    With $I_n$ is the $n\times n$-identity matrix, the determinant
    $\det(A(\mu))$ factors as  
    \begin{displaymath}
      \det(A(\mu)) = \det(D)\, \det(J(\bar c) D^{-1} - \mu I_n). 
    \end{displaymath}
  \item\label{item:A_char_poly_JDinv} 
    The determinant~$\det(A(\mu))$ is equal to the characteristic
    polynomial of the matrix $J(\bar c) D^{-1}$, up to the positive
    constant $\det(D)$.
  \item \label{item:A_not_zero}
    The polynomial~$\det(A(\mu))$ is not the zero polynomial.
  \end{enumerate}
\end{lemma}
\begin{proof}
  (\ref{item:A_factors}) follows by a simple computation and the
  fact that $D$ is a diagonal matrix with positive diagonal
  elements. (\ref{item:A_char_poly_JDinv}) follows directly from
  (\ref{item:A_factors}). Concerning (\ref{item:A_not_zero}), 
  we recall that right multiplication of a matrix with a diagonal
  matrix results in a rescaling of its column vectors. Hence right
  multiplication with such a diagonal matrix does not change its number of
  linear independent columns and hence its rank. Thus, $\rank(J(\bar
  c) D^{-1}) = \rank(J(\bar c))=s>0$ by assumption. Hence, in
  particular, the characteristic polynomial of $J(\bar c) D^{-1}$ is
  not the zero polynomial. 
\end{proof}

\begin{corollary}
  \label{coro:det_ev_A_nonzero}
  Let $\bar \mu\in\R$ be given. Then
  \begin{enumerate}[{(}i{)}]
  \item\label{item:det_A_nonzero}
    $\det(A(\bar\mu))=0$ if and only if $\bar\mu$ is an eigenvalue
    of~$J(\bar c)D^{-1}$.
  \item\label{item:ev_A_nonzero} If $\bar\mu$ is not an eigenvalue of
    $J(\bar c)D^{-1}$, then all eigenvalues of $A(\bar\mu)$ are
    nonzero.
  \end{enumerate}
\end{corollary}
\begin{proof}
  (\ref{item:det_A_nonzero}) follows by
  Lemma~\ref{lemma:det_JD_not_zero}(\ref{item:A_factors}) 
  as by assumption~$\det(D){\neq}0$. (\ref{item:ev_A_nonzero}) is a 
  direct consequence of (\ref{item:det_A_nonzero}).
\end{proof}

Because the determinant of a matrix is equal to the product of its
eigenvalues, the sign of $\det(A(\mu))$ can be used as an indicator
for the existence of positive  eigenvalues of the matrix $A(\mu)$: 

\begin{corollary}[$\sign(\det(A(\mu)))$ and positive eigenvalues of
  $A(\mu)$] \mbox{} \\ \vspace{-2ex}
  \label{coro:sign_A_evals}
  \begin{enumerate}[{(}I{)}]
  \item If $\mu$ is such that $\sign(\det(A(\mu)))=(-1)^{n+1}$, then
    \begin{enumerate}[{(}a{)}]
    \item the matrix $A(\mu)=J(\bar c)-\mu D$ has at least one positive
      eigenvalue and
    \item the number of positive eigenvalues of $A(\mu)=J(\bar c)-\mu D$
      is odd.
    \end{enumerate}
  \item If $\mu$ is such that $\sign(\det(A(\mu)))=(-1)^{n}$, then the
    matrix $A(\mu)=J(\bar c)-\mu D$ has either no positive  eigenvalues
    or an even number of such eigenvalues.
  \end{enumerate}
\end{corollary}

\subsection{Conditions for instabilities}
\label{sec:Turing-Condi}
Lemma~\ref{lemma:det_JD_not_zero}(\ref{item:A_not_zero}) 
together with the $n-s$ conserved masses imply that the polynomial
$\det(A(\mu))$ is of the form 
\begin{equation}
  \label{eq:CharPoly_JD}
  \det(A(\mu)) = \det(D)\, \mu^{n-s}\left( a_0 \mu^s + \ldots +
    a_{s-1} \mu + a_s 
  \right)\ .
\end{equation}
If all eigenvalues of $A(\mu)$ have negative real part, then
$\sign(\det(A(\mu))=(-1)^n$. If an odd number of eigenvalues is
positive, then $\sign(\det(A(\mu))=(-1)^{n+1}$. As we do not want that
$A(\mu)$ has positive eigenvalues for all positive values of $\mu$, we
want to have a sign change of the polynomial $\det(A(\mu))$ on the
positive real line. The, arguably, algebraically easiest way to
achieve this is a situation where the sign of the leading coefficient
and of the constant coefficient are different. In this case
$\det(A(\mu))$ will qualitatively look like the cartoons depicted in
Fig.~\ref{fig:poly_cartoon} (or their negative). And for $n$ odd,
values of $\mu$ where $\det(A(\mu))$ is positive corresponds to
$A(\mu)$ having positive eigenvalues (at least one). In case of $n$
even, it is values of $\mu$ where $\det(A(\mu))$ is negative that
correpond to $A(\mu)$ having positive eigenvalues.

\begin{figure}[!htb]
  \centering
  \begin{subfigure}[b]{0.48\textwidth}
    \centering
    \includegraphics[width=\textwidth]{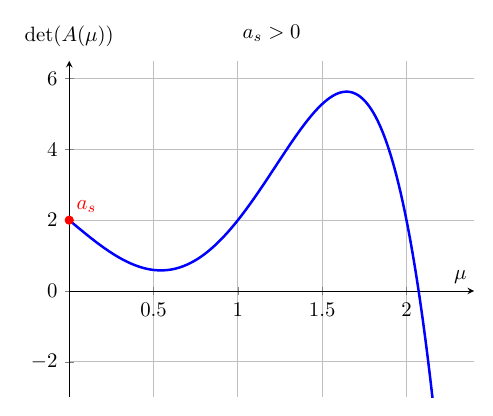}
    \caption{One positive root}
    \label{fig:poly1}
  \end{subfigure}
  \hfill 
  \begin{subfigure}[b]{0.48\textwidth}
    \centering
    \includegraphics[width=\textwidth]{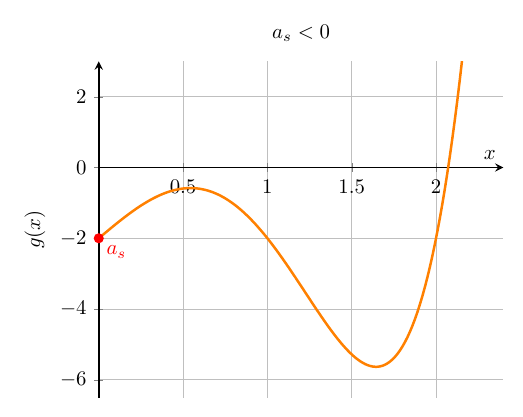}
    \caption{Three positive roots}
    \label{fig:poly2}
  \end{subfigure}
  \caption{
    Examples of polynomials with constant and leading
    coefficient of different sign that satisfy
    conditions~(\ref{eq:condi_a0}) \& (\ref{eq:condi_as}) for $n$ is
    odd ($n$ even, corresponds to the negative of the depicted
    curves).
  }
  \label{fig:poly_cartoon}
\end{figure}

Recall that the coefficients $a_0$, \ldots, $a_s$ are rational
functions of the parameters $k_i$, $d_i$ and $\xi_i$. In the
following, we then assume that there exist parameter values, such that
\begin{subequations}
  \begin{align}
    \label{eq:condi_a0}
    \sign(a_0) &=  (-1)^n\\
    \label{eq:condi_as}
    \sign(a_s) &= (-1)^{n+1}\ ,
  \end{align}
\end{subequations}

In this situation there exists (at least one) positive solution
to $\det(A(\mu))=0$. Let $\bar\mu$ be the smallest positive zero of
$\det(A(\mu))=0$.  In the following theorem we note explicit
conditions on the domain size $|\Om|$ for $\mu_1\in(0,\overline{\mu})$, 
and hence instability of $\oc$, where $\mu_1$ is the lowest positive 
eigenvalue $\mu_1$ of the (positive) Neumann Laplacian $-\Delta$. 
These give guidance for the numerical analysis of the bifurcating
branches in \S\ref{numsec}.

\begin{theorem}\label{theo:existence_mu_k}
  Recall the polynomial $\det(A(\mu))$ of (\ref{eq:CharPoly_JD}) and
  let $k_i$, $d_i$ and $\xi_i$ be such that the coefficients~$a_0$
  and~$a_s$ satisfy~(\ref{eq:condi_a0})~\&~(\ref{eq:condi_as}). As
  above, let~$\bar\mu$ be the smallest positive zero
  of~$\det(A(\mu)){=}0$.
  Assume that $\Omega\subset\R^d$, $d=1,2,3$ is simply connected
  with length, area, or volume $|\Omega|$. In this case there
  exists an eigenvalue~$\mu_{\ell}$ of the operator~$-\Delta$
  on~$\Omega$ such that the corresponding matrix $A_{\ell} =  J(\bar
  c) - \mu_{\ell} D$ has an odd number of positive eigenvalues (and at
  least one), if $|\Omega|$ satisfies
  \begin{align}
    &\text{$d=1$ ($\Omega \subset \R$ an interval)}:
      \label{eq:Omcrit-line}
      |\Omega| > \frac{\pi}{\sqrt{\bar\mu}}; \\
    &\text{$d=2$ ($\Omega\subset\R^2$):\quad}
      \label{eq:Omcrit_2D}
      |\Omega| > p_{1,1}^2 \frac{\pi }{\bar\mu}, 
      \intertext{
      where $p_{1,1}\approx 1.8412$ is first positive zero of the
      first derivative of the first Bessel function~$J_{1}$;
      }
    &\text{$d=3$ ($\Omega\subset\R^3$):}\quad 
      \label{eq:Omcrit_3D}
      |\Omega| > \frac{4}{3} p_{\frac{3}{2},1}^3
      \frac{\pi}{(\bar\mu)^{\frac{3}{2}}},
  \end{align}
  where $p_{\frac{3}{2},1}\approx 2.0816$ is first positive zero
  of the first derivative of $x^{-\frac{1}{2}} J_{\frac{3}{2}}(x)$
  for the Bessel function $J_{\frac{3}{2}}(x)$. 
\end{theorem}

\begin{proof} 
  Assumptions (\ref{eq:condi_a0}) and (\ref{eq:condi_as}) imply that
  $\sign(A(\mu))=\sign(a_s)$ on $(0,\bar\mu)$. By
  Corollary~\ref{coro:sign_A_evals} this implies that for every
  $\mu\in(0,\bar\mu)$ the corresponding matrix $A(\mu)= J(\bar c)-\mu D$
  has at least one positive eigenvalue. Hence, if the lowest nontrivial
  eigenvalue $\mu_{1}$ of $-\Delta$ on~$\Omega$ lies in $(0,\bar\mu)$,
  then the corresponding matrix $A_{\ell}=J(\bar c)-\mu_{\ell} D$ has an
  odd number of (and at least one) positive eigenvalues. For $d=1$ we
  have $\mu_1=\pi/|\Om|$, yielding \reff{eq:Omcrit-line}. For $d=2$ and
  $d=3$ the results follow by comparison with disks and balls (with
  again explicit eigenvalues), as by the isoperimetric inequality these
  yield the largest first eigenvalues under given $|\Om|$. In detail,
  for $d=2$ we have $\mu_1 \leq p_{1,1}^2 \frac{\pi}{|\Omega|}$
  \cite[eq.~(1.15)]{Ashbaugh1993}), yielding \reff{eq:Omcrit_2D},
  while $d=3$ yields  $\mu_1 \leq p_{\frac{3}{2},1}^2
  \left(\frac{\frac{4}{3}\pi}{|\Omega|}\right)^{\frac{2}{3}}$
  \cite[eq.~(1.15)]{Ashbaugh1993}, and hence \reff{eq:Omcrit_3D}. 
\end{proof}

If $\bar \mu$ is the only positive root of $\det(A(\mu))=0$ we can 
sharpen Theorem~\ref{theo:existence_mu_k}: 
\begin{corollary}
  In the setting of Theorem~\ref{theo:existence_mu_k}, let $\bar\mu$
  be the only positive root of $\det(A(\mu))=0$. Then conditions
  (\ref{eq:Omcrit-line}) -- (\ref{eq:Omcrit_3D}) are necessary and
  sufficient for the existence of an eigenvalue $\mu_{\ell}$ of $-\Delta$
  on~$\Omega$ such that the corresponding matrix~$A_{\ell}$ has an odd
  number of positive eigenvalues (and at least one).
\end{corollary}

\begin{proof}
  By Theorem~\ref{theo:existence_mu_k} conditions
  (\ref{eq:Omcrit-line}) -- (\ref{eq:Omcrit_3D}) are
  sufficient. To establish necessity let~$\mu_{\ell}$ be an eigenvalue of
  $-\Delta$ on~$\Omega$ such that the corresponding matrix $A_{\ell}$
  has an odd number of positive eigenvalues (and 
  at least one). In this case $\sign(\det(A(\mu_{\ell}))) =
  \sign(a_s)$ by Corollary~\ref{coro:sign_A_evals}. As, by assumption,
  $\bar\mu$ is the only positive root and~$a_0$ and~$a_s$
  satisfy~(\ref{eq:condi_a0}) \& (\ref{eq:condi_as}), it follows that
  $\sign(\det(A(\mu)))=\sign(a_s)$ if and only if $\mu \leq \bar\mu$
  (cf. the general form of the polynomial $\det(A(\mu))$
  in~(\ref{eq:CharPoly_JD})). Hence we conclude $\mu_{\ell} \leq
  \bar\mu$. The properties of the eigenvalues of~$-\Delta$ on~$\Omega$
  are such that $\mu_1 \leq \mu_{\ell}$ (see e.g.\
  \cite[Theorem~11.5.2]{jost2013partial} on the eigenvalues
  of~$-\Delta$). Hence $\mu_1<\bar \mu$ and
  conditions~(\ref{eq:Omcrit-line}) -- (\ref{eq:Omcrit_3D}) follow by 
  the same rearrangements as in the proof of
  Theorem~\ref{theo:existence_mu_k}.
\end{proof}

Based on Theorem~\ref{theo:existence_mu_k}, we can formulate
the following condition concerning Turing-like instabilities:
\begin{corollary}
  \label{coro:cond-turing-inst}
  Suppose there exist $k_i$, $\xi_i$ and $d_i$ such that $a_0$ and
  $a_s$ satisfy~(\ref{eq:condi_a0}) \&
  (\ref{eq:condi_as}). Further suppose that $\Omega$ is one of those
  discussed in Theorem~\ref{theo:existence_mu_k}. If the steady state
  $\bar c$ of the ODE~(\ref{eq:crn_ode}) is such that all
  nonzero eigenvalues of the matrix $J(\bar c)$ have negative real
  part, then Turing-like instabilities occur.
\end{corollary}
\begin{proof}
  By Theorem~\ref{theo:existence_mu_k} there exists at least one
  eigenvalue~$\mu_{\ell}$ of the operator~$-\Delta$ on~$\Omega$, such that
  the corresponding matrix $J(\bar c)-\mu_{\ell} D $ has an odd number of
  (and at  least one) positive  eigenvalues. As by assumption all
  nonzero eigenvalues $J(\bar c)$ have negative real part the
  conditions for Turing-like instabilities described are satisfied. 
\end{proof}

\begin{remark}
  \label{rem:L_nr_ev}
  Let $\Omega\subset\R$ be a line segment with fixed length $L$ and
  assume that $\bar\mu$ is the unique positive zero of
  $\det(A(\mu))=0$ (cf.\ Fig.~\ref{fig:poly2}). Then, under the
  assumptions (\ref{eq:condi_a0}) -- (\ref{eq:condi_as})) one has
  $\sign(\det(A(\mu)))=(-1)^{n+1}$ for $0<\mu<\bar \mu$ 
  and $\det(A(\mu))=(-1)^n$ for $\mu>\bar\mu$. Since the eigenvalues
  of $-\Delta$ on $\Omega$ are given by $\mu_{\ell} = \left(\frac{\ell
      \pi}{L}\right)^2$, we can then determine $L$ such that an
  arbitrary number of eigenvalues $\mu_{\ell}$ is such that
  $\det(A(\mu_{\ell}))>0$ by rearranging the above formula. For each
  index~$\ell$ with  
  \begin{displaymath}
    1 \leq \ell \leq \left\lfloor \frac{\sqrt{\bar \mu} L}{\pi}
    \right\rfloor 
  \end{displaymath}
  one has $\det(A(\mu_{\ell}))>0$ and the corresponding matrix~$J(\bar
  c)-\mu_{\ell} D$ has thus at least one positive eigenvalue (and
  always an odd number of such eigenvalues).
\end{remark}

\subsection{Application to the double phosphorylation 
mechanism}\label{sec:Turing_dd}

We now apply the above analysis to the network~(\ref{mapk}). We
evaluate the Jacobian of the right hand side of
(\ref{eq:ma_ode_1})--(\ref{eq:ma_ode_9}) at~(\ref{eq:mono_para}) to
obtain~$J(\bar c)$ and compute $\det(J(\bar c) D^{-1} -\mu I_9)$. This
yields a degree nine polynomial in~$\mu$ of the form 
\begin{displaymath}
  \mu^3(a_0 \mu^6 + \dots + a_5 \mu + a_6) =0,
\end{displaymath}
where the $a_i$ are rational functions of the $k_i$, $d_i$ and
$\xi_i$. All terms in the common denominator of the $a_i$ have
positive sign, hence the common denominator of the $a_i$ is 
positive for positive values of $k_i$, $d_i$ and
$\xi_i$. Consequently, in light of
Theorem~\ref{theo:existence_mu_k} the signs of the numerators 
of the coefficients $a_0$ and $a_6$ are of interest. We note the
following facts:
\newcounter{fact}
\begin{enumerate}[{Fact~}1{:}]
\item\label{item:fact_1}
  All terms in the numerator of $a_0$ have negative sign, hence
  $a_0$ is negative for all positive values of the $k_i$, $d_i$ and
  $\xi_i$.
\item The numerator of $a_6$ as a polynomial in $\xi_1$ is of degree
  two (in~$\xi_1$) that factors as 
  \begin{displaymath}
    \text{Numerator}(a_6) = \left(d_{3} d_{8} k_{12} k_{6}-d_{5} d_{9}
      k_{3} k_{9}\right) \gamma_0 \xi_1^2 + \gamma_1 \xi_1 + \gamma_2,
  \end{displaymath}
  where $\gamma_0$, $\gamma_1$ and $\gamma_2$ are polynomials in the 
  $k_i$, $d_i$ and $\xi_2$ and $\xi_3$.
\item The coefficient~$\gamma_2$ contains only negative terms,
  $\gamma_1$ contains positive and negative terms and~$\gamma_0$
  contains only positive terms.
  \setcounter{fact}{\value{enumi}}
\end{enumerate}
To further discuss $a_6$ as a function of $\xi_1$ we use the symbol
$a_6(\xi_1)$. Note that for arbitrary positive values of the $k_i$ and
$d_i$ that satisfy~(\ref{eq:d_k_ineq}) and arbitrary positive $\xi_2$,
$\xi_3$ the signs of the terms in $\gamma_0$ and $\gamma_2$ imply that
$a_6(\xi_1)$ is a quadratic polynomial in~$\xi_1$ with negative
constant coefficient. Thus, we observe: 
\begin{enumerate}[{Fact~}1{:}]
  \setcounter{enumi}{\value{fact}}
\item\label{item:fact_4} If
  \begin{equation}
    \label{eq:d_k_ineq}
    d_{3} d_{8} k_{12} k_{6}-d_{5} d_{9} k_{3} k_{9} >0,
  \end{equation}
  then there exist a unique positive value $\bar \xi_1$ such that
  $a_6(\bar\xi_1)=0$.
\item\label{item:fact_5} This unique positive zero $\bar \xi_1$ of
  $a_6(\xi_1)$ is such that
  \begin{displaymath}
    a_6(\xi_1) < 0\text{ for $0 < \xi_1< \bar\xi_1$ and } a_6(\xi_1) >
    0\text{ for $\xi_1>\bar\xi_1$}\ .
  \end{displaymath}
  \setcounter{fact}{\value{enumi}}
\end{enumerate}
This leads to the final conclusion:
\begin{enumerate}[{Fact~}1{:}]
  \setcounter{enumi}{\value{fact}}
\item\label{item:fact_6} If $k_i$, $d_i$ satisfy~(\ref{eq:d_k_ineq})
  and if $\xi_1 > \bar \xi_1$, then the coefficients $a_0$ and $a_6$
  of $\det(J(\bar c) - \mu D)$ satisfy the conditions
  (\ref{eq:condi_a0}) \& (\ref{eq:condi_as}) and we are in the
  situation of Theorem~\ref{theo:existence_mu_k}.
\end{enumerate}

\begin{lemma}[Eigenvalues $\mu_{\ell}$ of $-\Delta$ and 
  positive eigenvalues of~$J(\bar c) - \mu_{\ell} D$
  for appropriate $\Omega$]
  \label{lem:eigenvalues_Jc_muD_dd}
  Pick $k_i$ and $d_i$ such that
  \begin{equation}
    \label{eq:1}
    \frac{k_3 k_9}{k_6 k_{12}} < \frac{d_3 d_8}{d_5 d_9}\ .
  \end{equation}
  Fix $\xi_2>0$ and $\xi_3>0$, let $\bar\xi_1$ be the
  unique positive solution to $a_6(\xi_1)=0$ and pick
  $\xi_1>\bar\xi_1$. Determine $\bar c$ according to
  (\ref{eq:mono_para}). \\ 
  If $\Omega$ is as in Theorem~\ref{theo:existence_mu_k}, then the
  operator $-\Delta$ has an eigenvalue $\mu_{\ell}$ such that the
  corresponding matrix~$J(\bar c)-\mu_{\ell} D$ has an odd number of
  positive eigenvalues (and at least one).
\end{lemma}
\begin{proof}
  Inequality~(\ref{eq:1}) is equivalent to~(\ref{eq:d_k_ineq}) and as
  explained in Fact~\ref{item:fact_1} -- \ref{item:fact_6} this
  implies that the $k_i$, $d_i$ and $\xi_i$ are such that the
  coefficients $a_0$ and $a_6$ of $\det(J(\bar c)D^{-1}-\mu I_6)$
  satisfy the assumptions~(\ref{eq:condi_a0}) \&
  (\ref{eq:condi_as}). Hence, by Theorem~\ref{theo:existence_mu_k} the
  domain~$\Omega$ can be 
  chosen such that there exists an eigenvalue $\mu_{\ell}$ of the
  operator~$-\Delta$ such that the matrix $J(\bar c)-\mu_{\ell} D$ has
  and odd number of positive eigenvalues (and at least one).
\end{proof}

The ODEs derived from network~(\ref{mapk}) can admit
multistationarity~\cite{Conradi2008}, that is, there may be at least
two positive steady states for some values of the total amounts. We
expand on this in the following remark.

\begin{remark}\label{msrem}
[Multistationarity in the ODE system of
  network~(\ref{mapk})]
  As discussed in \S\ref{sec:odes-pdes-MA-network}, in the ODE
  setting the solutions $c(t)$ are confined to the affine subspaces
  $Z^{'} c = Z^{'}c_0$ (where $c(0)=c_0$), i.e., with masses $m=m(c_0)$. 
  This has important implications when talking about the number of
  steady states: Given a monomial parametrization, there exists an
  infinite number of steady states.
  But given an initial condition $c_0>0$ only a limited number of
  steady states is of interest, those with $m(\overline{c})=m_0$.
  
  More formally, let $\bar c(\xi)$ be the steady state parametrization
  (suppressing the $k_i$ dependency for notational simplicity) and let
  $m_0$ be given. Then only those elements of $\bar c(\xi)$ contained
  in the affine subspace $Z^{'} c= Z^{'} c_0$ are relevant for the
  solution $c(t)$ with $c(0) = c_0$. These are defined by the positive
  solutions of
  \begin{displaymath}
    Z^{'} \bar c(\xi) = Z^{'} c_0\ . 
  \end{displaymath}
  If there are $\geq 2$ such positive solutions, then one speaks of
  multistationarity, and if $\ge 2$ such steady states are locally
  dynamically stable, then of multistability. In~\cite{maya-bistab} it
  has been shown, that if the rate constants of network~(\ref{mapk})
  satisfy 
  \begin{equation}
    \label{mcon}
    k_{3} k_{9} - k_{12} k_{6}<0,
  \end{equation}
  then there exists $\bar \xi_1$ (for every fixed positive $\bar
  \xi_2$ and $\bar \xi_3$) such that the equation 
  \begin{displaymath}
    Z^{'} \bar c(\xi) = Z^{'} \bar c (\bar \xi)
  \end{displaymath}
  has three positive
  solutions~\cite[Theorem~5.1\&Corollary~5.1]{maya-bistab}. (Here the 
  steady state $\bar c (\bar \xi)$ takes on the role of the vector
  $c_0$  in defining the affine subspace.) 
  In the numerical examples in \S\ref{numsec} we see the
  multistationarity condition \reff{mcon} approximately, as we do not
  aim for sharp $\bar\xi_1$, see Figs \ref{f1} and \ref{f1b}, and
  Footnote \ref{mcfoot}.
\end{remark}

\section{Numerical examples}\label{numsec}

\subsection{General remarks and setup}\label{numisec}
We use the numerical continuation and bifurcation package 
\pdep\ \cite{p2pbook,p2phome} to first continue steady states 
and some time-periodic orbits (POs) of 
the reaction--diffusion PDE \reff{eq:pde}, i.e., 
\begin{gather}
  \label{rd1}
  \pa_t c=G(c):=f(c)+D\Delta c,\quad f(c)=\Gamma r(k,c),\quad 
  D=\diag(d_1,\ldots,d_9),  
\end{gather}
on domains $\Om=(-l,l)\subset\R$ (1D interval) or 
$\Om_R$ (disk of radius $R$) with Neumann BCs. Subsequently we will also 
run some numerical time integration, aka direct numerical simulation (DNS). 
We throughout fix the ``base parameter set'' 
\begin{subequations}
  \label{bp1}
  \begin{align}
    &\text{$(d_1,d_2,d_3,d_4,d_5,d_6,d_7,d_8,d_9)=
      (0.1, 0.3, 2, 0.4, 0.5, 0.02, 0.8, 0.5, 0.5)$, and}\\
    &\text{$(k_1,k_2,k_4,k_5,k_6,k_7,k_8,k_9,k_{10},k_{11},k_{12})
      =(1, 0.3, 4.2, 1.6, 1, 0.1, 2.2, 0.5, 0.5, 0.8, 1)$, }
  \end{align}
\end{subequations}
and use $k_3>0$ as a bifurcation parameter, and $\overline{c}(\xi)$ 
with $(\xi_1,\xi_2,\xi_3)=(1,1,2)$ as a starting point for the continuation. 
Additionally, we use \reff{bp1} with $k_9$ changed to 
\begin{align}
  \label{bp1b}
  \text{$k_9=1$,} 
\end{align}
because in this case the primary loss of stability of $\oc$ 
is due to Hopf bifurcations.

The steady state continuation and bifurcation requires to 
explicitly implement as constraints the mass conservations 
\begin{align}
  \label{mc1}
  \ddt m_i=0,
\end{align}
$m_1=\spr{c_2+c_3+c_5},\ m_2=\spr{c_7+c_8+c_9}, \ 
m_3=\spr{c_1+c_3+c_4+c_5+c_6+c_8+c_9}$, where 
$\spr{c}=\frac 1 {|\Om|}\int_\Om c\dd x$, 
because for numerical continuation and bifurcation we must get 
rid of the three 0 eigenvalues associated to \reff{mc1} 
(see \cite[\S3.5]{p2pbook} for general background on continuous 
symmetries and constraints in \pdep). 
For this, there are basically two options: Either we choose 
\begin{align}
  \label{q1}
  q(c):=\bpm q_1(c)\\q_2(c)\\q_3(c)\epm:=
  \bpm \spr{c_2+c_3+c_5}-m_1(k_3)\\
  \spr{c_7+c_8+c_9}-m_2(k_3)\\
  \spr{c_1+c_3+c_4+c_5+c_6+c_8+c_9}-m_3(k_3)\epm=\bpm 0\\0\\0\epm, 
\end{align}
with $m_j(k_3)$ given from the monomial parametrization (MP) 
\reff{eq:mono_para} of that specific homogeneous branch, or 
\begin{align}
  \label{q2}
  q(c):=\bpm q_1(c)\\q_2(c)\\q_3(c)\epm:=\bpm \spr{c_2+c_3+c_5}-m_1\\
  \spr{c_7+c_8+c_9}-m_2\\
  \spr{c_1+c_3+c_4+c_5+c_6+c_8+c_9}-m_3\epm=\bpm 0\\ 0\\0\epm, 
\end{align}
with $m=(m_1,m_2,m_3)$ set at initialization. The MP variant 
is close to the analysis and explicitly based on the 
homogeneous branch \reff{eq:mono_para}, and only allows 
bifurcations from this with $m=m(k_3)$. We use this in Figures \ref{f1} and 
\ref{f1b}. The second variant 
can be considered as constraints given by initial conditions 
(of the continuation), and hence 
yields bifurcation diagrams (in $k_3$ or any other parameter) 
at fixed $m$ throughout. For brevity we call this a ``natural 
parametrization'' (NP), and use this in Figures \ref{f1c} and \ref{f2a}.

In both cases, to implement the constraints \reff{q1} or \reff{q2} 
we introduce three Lagrange--multipliers $\eps_1, \eps_2, \eps_3$ and 
obtain the modification 
\begin{align}
  \wt f(c)=(f_1{+}\eps_3,f_2{+}\eps_1,f_3{+}\eps_1{+}\eps_3,f_4{+}\eps_3,f_5{+}\eps_1{+}\eps_3,f_6{+}\eps_3,f_7{+}\eps_2,f_8{+}\eps_2{+}\eps_3, f_9{+}\eps_2{+}\eps_3), 
\end{align}
which we directly rename to $f$. When solving $G(c)=0$ with \reff{q1} or 
\reff{q2} with the new $f$ 
we then want the $\eps_i$ to stay zero (numerically $\CO(10^{-12})$, say), 
which they do. 

When considering \reff{rd1} on a disk, we have a 
continuous rotational symmetry (see also Remark \ref{pcrem}): 
In polar coordinates $x=r(\cos\phi,\sin\phi)$, if $c(r,\phi)$ 
is a steady state, so is $c(r,\phi+\vt)$ for any $\vt\in\R$. 
Hence, if $\pa_\phi c\not\equiv 0$, then $\pa_\phi c$ is in the 
kernel $N(\pa_c G)$ of $\pa_c G$, yielding a fourth 0 eigenvalue. 
To remove this, we go into a co--rotating frame $\vt=\phi-st$, 
which modifies \reff{rd1} to 
\begin{align}
  \label{rd2}
  \pa_t c=G(c)+s\pa_\phi c, 
\end{align}
with the speed $s$ a priori unknown. To continue solutions of \reff{rd2}, 
additional to $k_3$ we  ``free'' the parameter $s$ and then need to 
add the rotational phase condition (PC) 
\begin{align}
  \label{cs} 
  \spr{\pa_\phi u,u_0}=0, 
\end{align}
where $u_0$ is a reference solution, usually from the previous 
continuation step.%
\footnote{More formally, this makes the continuation orthogonal to the 
group orbit of rotations, again see \cite[\S3.5]{p2pbook} for general 
background, or \cite[Part II]{p2pbook} or \cite{symtutb} for simple examples.}
\footnote{See also \cite{SI} for the \mlab\ sources 
of our implementation, together with a reference sheet 
and further plots.} 
Steady states of \reff{rd2} with $s=0$ are genuine 
steady states of \reff{rd1}, and 
steady states of \reff{rd2} with $s\ne 0$ (so called 
relative equilibria) correspond to states rotating in 
\reff{rd1} with angular speed $s$, so called rotating waves 
(RWs), and as we shall see these RWs play a prominent role 
in \reff{rd1} over disk domains. 

\begin{remark}
  \label{pcrem}
  We collect a number of Remarks on the spectra 
  and eigenfunctions of 
  $\pa_c G(\overline{c})$ on intervals and disks with Neumann BCs 
  as already used in Theorem \ref{theo:existence_mu_k}, and on the 
  ``symmetry perspective'' \cite{gs85,GS2002,Hoyle} 
  on bifurcations from homogeneous steady states $\overline{c}$. 

  a) Over $\R$, the eigenfunctions 
  of $\pa_c G$ have the form of {\em Fourier modes}  
  $V(x;\kap)=\er^{\ri \kap x}\Vh(\kap)$ where $\kap\in\R$ is called the 
  wave number, and $\Vh(\kap)\in\C^9$ is an eigenvector of 
  $A(\kap):=J(\oc)-\kap^2 D$. For any $\kap\in\R$,  $A(\kap)\in\R^{9\times 9}$ 
  has $9$ (counted with multiplicity) 
  eigenvalues $\mu_1(\kap),\ldots,\mu_9(\kap)\in\C$, and we 
  define the curves $\kap\mapsto \mu_j(\kap)$ by sorting wrt 
  decreasing real parts. We can then numerically compute 
  and plot the {\em dispersion relations} (DRs), 
  i.e., the leading eigenvalue curves $\kap\mapsto \mu_j(\kap)$, 
  $j=1,2,3$, say, as in Fig.\ref{f1}(a,b) and Fig.\ref{f1b}(a,b). 
  
  Over bounded intervals with Neumann BCs, i.e., $\Om=(-l,l)$, only 
  wave numbers $\kap=\kap_m=m\pi/(2l)$ are allowed. 
  Moreover, \reff{rd1} is 
  $Z_2$ equivariant, i.e., invariant under $x\mapsto -x$. Therefore: 
  \begin{compactitem}
  \item All steady state BPs from $\oc$, i.e., $\mu_j(\kap_m)=0$ for 
    some $j$, are generically simple, with kernels of the form 
    $V(x)=\sin(m \pi x/(2l))\Vh_m$, 
    $\Vh_m\in\R^9$, $m\in\N$ odd, or 
    $V(x)=\cos(m \pi x/(2l))\Vh_m$, $m\in\N$ even. For $m\ne 0$ 
    the bifurcating branches are necessarily pitchforks due 
    to the ``hidden symmetry'' of shifting the solution by half 
    a spatial period. However, for 
    $m=0$ we shall find a transcritical bifurcation of a second 
    homogeneous branch $\tilde{\oc}$ from $\oc$, at ``small'' $k_3$. 
    This should be expected from the multistationarity condition 
    \reff{mcon}, see Footnote \ref{mcfoot}. 
  \item We have Hopf bifurcation points if a pair of complex 
    conjugate eigenvalues $\mu_j=\overline{\mu_{j+1}}$ with $\im(\mu_j)\ne 0$ 
    crosses the imaginary axis, see for instance Fig.\ref{f1b}(a,b). 
    In 1D with NBCs, this yields the pitchfork bifurcations of 
    so called standing waves, i.e., time periodic orbits with 
    discrete points $x_j$ where $c(t,x_j)$ oscillates with 
    maximal amplitude. See Fig.\ref{f1b} for examples, where 
    however we only briefly look at such SWs 
    because their computation is quite demanding numerically, 
    and the SWs obtained are all unstable, possibly with some exception in a 
    narrow parameter regime. Instead we will focus on RWs in 
    2D, see Fig.\ref{f2a}. The analog of RWs in 1D would be 
    traveling waves (TWs), which for instance would bifurcate 
    for periodic boundary conditions together with the SWs, but here 
    the TWs are prohibited by the Neumann BCs. 
  \end{compactitem}
  
b) The eigenfunctions of the (scalar) Neumann Laplacian on disks 
$D_R$ of radius $R$ are 
\begin{align}
  \label{dnl}
  v(r,\phi)=b(r)\er^{\ri m\phi},\quad m=0,1,\ldots, 
\end{align}
where $b(r)=J_m(s_{mn}'r/R)$ with $J_m$ the $m$th Bessel function 
of the first kind and $s_{mn}'=n$--th zero of $J_m'(r)$, and 
$m$ is called azimuthal wave number. We thus label 
the eigenfunctions as $v_{mn}(r,\phi)$, with eigenvalues 
$\mu_{mn}=(s_{mn}'/R)^2$, cf.~Theorem \ref{theo:existence_mu_k}. 
For the (complexified) 
vector valued linearized problem $G_c(\oc)V=\mu V$, $V\in L^2(\Om,\C^9)$, 
the eigenfunctions have the form 
$$V(r,\phi)=v_{mn}(r,\phi)\Vh_{mn} \text{ with }\Vh_{mn}\in \C^9.
$$ 
The rotational invariance of \reff{rd1} on disks 
yields that the problem is $O(2)$ equivariant, i.e., 
invariant under $\phi\mapsto -\phi$ and under 
$\phi\mapsto \phi+\vt$, $\vt\in\R$. Therefore: 
\begin{compactitem}
\item Steady state BPs with kernels $N(\pa_c G(\oc))$ 
with angular dependency ($m\ne 0$) come with multiplicity at least 
two, i.e., $v_{mn}(r,\phi)\Vh_{mn}$ and 
$v_{mn}(r,\phi+\pi/2)\Vh_{mn}$ are two linearly independent eigenfunctions. 
Moreover, all bifurcations of this type must be pitchforks due 
to the ``hidden'' symmetry $\phi\mapsto \phi+\pi/2$. For $m=0$ 
(``Mexican hats'') 
the BPs are generically simple, and we expect transcritical bifurcations 
of such branches. 
\item For Hopf bifurcations with $m\ne 0$ we automatically have 
three bifurcating branches:  clockwise ($s<0$ in \reff{rd2}) and 
counterclockwise ($s>0$ in \reff{rd2}) RWs, RW$_+$ and RW$_-$ for 
short, and again SWs as 
equal amplitude superpositions of RW$_{\pm}$. 
We refrain from computing SWs in 2D, but in \S\ref{2dsec} we 
compute branches of 2D RWs. 
\end{compactitem}
\end{remark}
\begin{remark}
  \label{lwrem} 
a) 
Inspection of the DRs in Fig.\ref{f1}(a,b), and of similar 
DRs in various other parameter regimes, yields that the 
instabilities of $\oc$ wrt modes $V(x;\kap)$ with $\kap\ne 0$ 
is of {\em long wave} (or {\em sideband}) type, as analyzed in 
Theorem \ref{theo:existence_mu_k}. 
 That means that 
the instability sets in by a change of curvature 
of the critical eigenvalue curve $\kap\mapsto \mu_1(\kap)$ 
at $\kap=0$. We have $\mu_1(0)=0$ (and $\mu_2(0)=\mu_3(0)=0$ 
for all parameters due to the mass conservations), 
$\mu_j'(0)=0$ 
for $j=1,\ldots,9$ due to the reflection symmetry $x\mapsto -x$, 
and $\re\mu_j(\kap)\sim -\kap^2$ for $\kap\to\infty$ and $j=1,\ldots,9$. 
Moreover $\re \mu_1''(0)<0$ for $k_3>k_{3,c}\approx 4.49$, 
and $\re \mu_1''(0)>0$ for $k_3<k_{3,c}$, yielding $\re \mu_1(\kap)>0$ in 
a small interval $(0,\kap_0)$.   

In this situation, over finite intervals $(-l,l)$, the first mode 
to become unstable always is the one with minimal wave number 
$\kap_1=\pi/(2\ell)$, followed by the second mode $\kap_2=\pi/\ell$, 
and so on. Therefore, these instabilities are called Turing--like, 
while genuine Turing instabilities are of finite wave number type, 
i.e., the critical wave number $\kap_c$ (defined over $\R$) 
is of order 1. 

b) Similar remarks also apply to the 2D case, where analogs of
Fig.\ref{f1}(a,b)  
can also be computed, yielding (punctured) disks 
$\{\kap=(\kap_1,\kap_2): |\kap|<\kap_0, \kap\ne 0\}$ of unstable 
wave vectors, respectively for finite domains 
discretization of these, for instance according to \reff{dnl}, 
and over disks we then find the order of instabilities as 
$(m,n)=(1,1), (2,1), (0,1), \ldots$. 

c) Examples of genuine Turing instabilities ($\kap_c{=}\CO(1)$) 
and Turing--Hopf instabilities ($\re\mu_1(\kap_c){=}0$ and 
$\im\mu_1(\kap_c)\ne 0$) in systems 
with conserved quantities can be found in, e.g., 
\cite{MC00} (for the conserved Swift--Hohenberg equation as a model 
problem), and \cite{HF18,YFB22,cylpre}. In order to find genuine 
Turing instabilities for \reff{rd1}, 
besides using the base parameter set \reff{bp1} we numerically computed 
DRs for various different parameter combinations, in particular 
varying the diffusion constants from (\ref{bp1}b). This was 
mainly based on intuition, trying to find  activator and 
inhibitor groups as in \cite{SMM00}, but without success, i.e., 
in all cases the instabilities were of long wave sideband type. 
Thus, finding genuine Turing instabilities for \reff{rd1} 
(and possibly associated wave instabilities with $\re\mu_1(\kap_c)=0$ 
and $\im\mu_1(\kap_c)\ne 0$) remains an interesting open problem. 
\end{remark}

\begin{figure}[ht]
  \centering
  \includegraphics[width=0.4\linewidth]{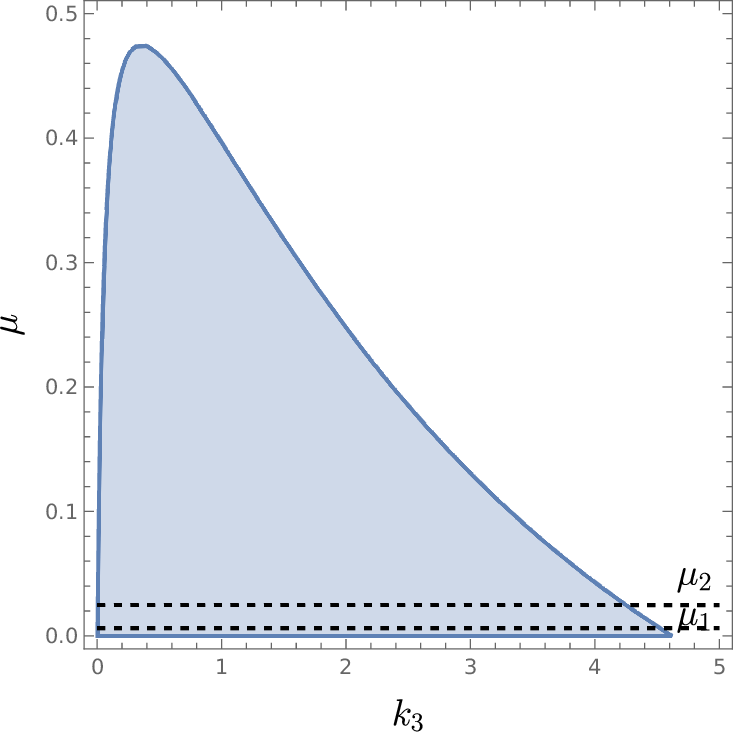}
  \caption{\label{fig:detA_pos}
    Shaded region: values of 
    $k_3$ and $\mu$ with $\det(A(\mu))>0$. 
  }
\end{figure}

\subsection{Numerical evaluation of the analytical instability 
predictions} 
For the base parameter set \reff{bp1}, 
the conditions (\ref{eq:1}) and~(\ref{eq:kc_Km_d_ineq}) evaluate to 
$ \frac{k_3 k_9}{k_6 k_{12}} = \frac{k_3}2 \text{ and } \frac{d_3 d_8}{d_5
    d_9} = 4$, i.e.,  $2 <  k_3 < 8$, 
and for \reff{bp1b} to $1<  k_3 < 4$. 
In light of Facts \ref{item:fact_4}~and~\ref{item:fact_5} on page 
\pageref{item:fact_4} we investigate 
$\det(A(\mu))$ and the coefficient $a_6$, namely 
\begin{equation}
  \label{eq:detA_hannes}
  \begin{split}
   \det(A)(\mu) &= C\mu^3(a_0\mu^6+\ldots+a_5\mu+a_6)\\
&=\frac{\mu^3}{(k_3+0.3)^2} \bigg[
    -\left(0.000048 k_3^2+0.0000288 k_3+4.32\cdot 10^{-6}\right) \mu^6 \\
    &\qquad
    - \left(0.000024 k_3^3+0.01113 k_3^2+0.00402467 k_3+0.000206352\right) \mu^5 \\
    &\qquad 
    - \left(0.00539418 k_3^3+0.716149 k_3^2+0.105762 k_3+0.00356709\right) \mu^4\\
    &\qquad
    -\left(0.339791 k_3^3+9.28858 k_3^2+1.02013 k_3+0.0268889\right) \mu^3\\
    &\qquad
    -\left(4.4146 k_3^3+33.5369 k_3^2+3.46057 k_3+0.0846663\right) \mu^2 \\
    & \qquad
    -\left(15.4912 k_3^3+16.2052 k_3^2+2.95065 k_3+0.0910211\right) \mu \\
    &\qquad
    -5.6007 k_3^3+25.6708 k_3^2+0.922778 k_3-0.00472212
    \biggr]. 
  \end{split}
\end{equation}
For positive $k_3$, all coefficients are negative, except
\begin{displaymath}
  a_6(k_3) =       -5.6007 k_3^3+25.6708 k_3^2+0.922778 k_3-0.00472212
\end{displaymath}
which is positive for $0.00454356 < k_3 < 4.61914$. 
As $a_0$ is negative for all positive values
of~$k_3$, the coefficients of the polynomial~(\ref{eq:detA_hannes})
satisfy the assumption~(\ref{eq:condi_a0})~\&~(\ref{eq:condi_as}) for
values of $k_3$ in the above interval. We can thus apply
Corollary~\ref{coro:det_ev_A_nonzero} and
Theorem~\ref{theo:existence_mu_k}. By Corollary~\ref{coro:sign_A_evals} 
eigenvalues $\mu_{\ell}$ of~$-\Delta$ such that $\det(A(\mu_{\ell}))>0$, 
yield instabilities, and in summary 
we find that $\det(A(\mu))>0$ for the values depicted in 
Fig.~\ref{fig:detA_pos}. For $\Om=(-l,l)$ with $l=20$ as in 
\S\ref{1dsec}, $\mu_1=\pi^2/(2l)^2\approx 0.0062$, 
$\mu_2\approx 0.0247$, and inspecting the associated lines 
in Fig.\ref{fig:detA_pos} indeed yields good predictions for the $k_3$ 
values for the bifurcations of these modes from $\oc$. 

\subsection{1D}\label{1dsec}
Fig.\ref{f1}(a,b) show DRs (cf.~Remark \ref{pcrem}a)) 
for linearizations of \reff{rd1} around the homogeneous 
steady states $\oc(k_3)$ from \reff{eq:mono_para}, 
with parameters from \reff{bp1}. These show 
that the primary instability of $\oc(k_3)$ is at some $k_3=k_{3c}\in(4,4.5)$ 
($k_{3c}\approx 4.48$) 
and is of steady long wave type. Even though the 
eigenvalues $\mu$ with small $|\Re(\mu)|$ are in general 
complex, the instability onset  
is for real eigenvalues, as illustrated in (b), and 
altogether there are no Hopf bifurcations at all from $\oc$ 
in this regime.  
(This is different in Fig.\ref{f1b} for $k_9=1$, where there are 
several Hopf bifurcations from $\oc$, including the primary loss 
of stability).

\begin{figure}[ht]
    \includegraphics[width=0.9\textwidth]{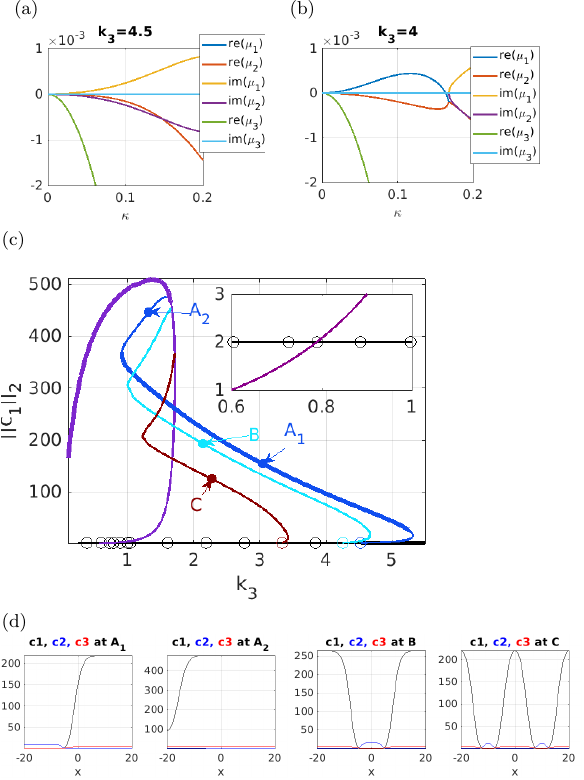}
\caption{{\small Parameters \reff{bp1}, 
constraints $m=m(k_3)$ based on \reff{eq:mono_para} (MP setting), 
continuation 
in $k_3$. (a,b) Dispersion relations before and after 
spatial instability of $\oc(k_3)$. (c,d) Basic bifurcation 
diagram, and sample solutions. \label{f1}}}
\end{figure}

Fig.\ref{f1}(c) shows a basic BD for \reff{rd1}, \reff{q1} on 
$\Om=(-20,20)$ with Neumann BCs%
\footnote{we refer to Remark \ref{numrem} and \cite{SI} for details 
on the spatial discretization}, where 
\begin{gather}
  \label{c1n}
  \|c_1\|_2:=\left(\frac 1 {|\Om|}\int_\Om c_1^2(x)\dd x
  \right)^{1/2}. 
\end{gather}
This is not a norm, as it uses only the first 
component of $c\in L^2(\Om,\R^9)$, but motivated by the fact that 
$c_1(k_3)=\xi_1$ independent of $k_3$ in the MP setting. 
Similarly, in the sample plots in (d) we restrict to the 
first three components of $c$, and refer to \cite{SI} for 
plotting other components. 

Thicker lines in (c) indicate stable (parts of) branches, and 
 branch points from the black branch $k_3\mapsto \oc(k_3)$ with 
$\|c_1\|_2\equiv \xi_1=2$ are indicated by $\circ$. We show 
the 1st (``{\tt A}'' branch, blue), 
2nd (``{\tt B}'' branch, light blue), and 4th (``{\tt C}'' branch, brown) 
patterned branches bifurcating from $\oc$. All these bifurcations are 
subcritical pitchforks (meaning that they initially 
go into direction of ``more stable'' $\oc$, i.e., towards larger $k_3$), 
with a turning point (fold) shortly after bifurcation.  
In particular, the primary patterned branch {\tt A} is unstable at bifurcation 
but stabilizes shortly after its right fold. 
Additionally we have 
a second homogeneous branch $\wt{\oc}$ (violet) bifurcating transcritically 
(see inset) from $\oc$ at $k_3\approx 0.8$, which for 
initially increasing $k_3$ shows a fold at $k_3\approx 1.9$ 
and then stabilizes after reconnection of the {\tt A} branch 
at $k_3\approx 1.8$.%
\footnote{\label{mcfoot}This behavior agrees with 
the multistationarity condition \reff{mcon}. For 
the parameter set \reff{bp1}, 
$k_{3\text{mult}}:{=}k_6k_{12}/k_9{=}2$, 
and we should expect the bifurcation of the $\wt{\oc}$ branch 
at some $k_3<k_{3\text{mult}}$, 
and $\wt{\oc}$ to extend to roughly  $k_{3\text{mult}}$, but with  in general 
the fold at some $k_{3f}<k_{3\text{mult}}$; this is because we do 
not vary $\xi_1$ here, i.e., do not search for the ``optimal'' 
$\bar\xi_1$ to yield multistationarity for $k_3$ near $k_{3\text{mult}}$.}

Similarly, also the other patterned branches connect $\oc$ and $\wt{\oc}$. 
Thus, besides the multistationarity of $\oc$ and various patterned 
branches for $k_3<k_{3,c}$, we also have  multistationarity of two 
homogeneous branches for $k_3<1.9$. Moreover, we have different 
multistabilities, namely of $\oc$ and {\tt A} for $k_3>k_{3c}$ and 
of $\wt{\oc}$  
and {\tt A} for $k_3\in(0.9,1.8)$ (left fold of A and reconnection of A to 
$\wt{\oc}$). 

In Fig.\ref{f1b} we illustrate that besides the steady state bifurcations 
as in Fig.\ref{f1}, \reff{rd1} can also show Hopf bifurcations. 
Namely, for $k_9=1$, the primary  instability of $\oc$ is at $k_3\approx 2.82$ 
and is of long wave Hopf type, with small imaginary parts, see (a). 
Subsequent bifurcations from $\oc$ then also include the 
steady long wave type, see (b) for the DR at smaller $k_3$. 
In the main BD in (c) we focus on steady branches, which we plot with 
the same colors as in Fig.\ref{f1}, and which behave completely analogous, 
and thus we omit samples here. The inset of (c) shows in more detail 
the behavior at loss of stability of $\oc$, including 
the first SW branch (orange, samples in (d)), and a ``breather'' branch 
(green, samples in (f)) 
which bifurcates from the A branch yielding gain of stability of A.  Here, 
\begin{gather}
  \label{c1nT} 
  \|c_1\|_2:=\left(\frac 1 {|\Om|T}\int_0^T\int_\Om c_1^2(t,x)\dd x\dd t
  \right)^{1/2}
\end{gather}
generalizes \reff{c1n} to $T$ periodic solutions. 
The behavior of $T$ along the SW and the 
breather branches is shown in (e). 
The SW branch again bifurcates subcritically (to 
increasing $k_3$), and along the branch, $T$ first decreases and 
then increases again. 
At the maxima of the amplitudes 
(at $x=\pm l$) the oscillations quickly develop into a rather strong 
relaxation type, which makes the continuation of the SW branch to large 
amplitude numerically expensive due to the fine $t$ discretization 
of up to 200 points needed, and the same holds for the breather branch.

\begin{figure}[ht]
  \includegraphics[width=0.9\textwidth]{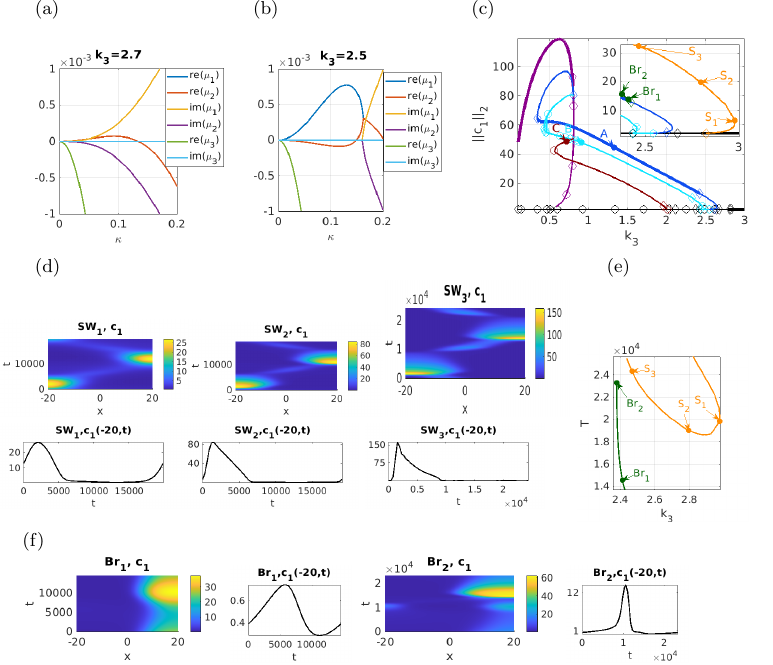}
  \vspace{-3mm}
\caption{{\small Like Fig.\ref{f1} but with $k_9=1$. 
Loss of stability of $\oc(k_3)$ now via Hopf bifurcations, $\diamond$ symbols. 
(a) DRs after first two Hopf bifurcations; (b) after first steady bifurcation. 
(c) BD; the zoom inset also shows the first SW branch (orange), 
and a breather branch (green). (d) space--time plots 
(top) of $c_1$ for the three SWs marked in (c), and behavior of $c_1(t)$ 
at the left boundary.  (e) behavior of period $T$ along Br and SW branches. 
(f) Samples from the breather branch. 
\label{f1b}}}
\end{figure}

\begin{remark}\label{numrem}
  a) For the computation of Fig.~\ref{f1}(c) and 
  Fig.\ref{f1b}(c) we use 
  piecewise linear finite elements in space with a rather moderate 
  discretization of 100 mesh points, because of the rather fine 
  meshes in $t$ of up to 200 points needed for the computation of SWs. 
  Due to the 9 components of $c$, this yields $9\times 100\times 200=180000$ 
  degrees of freedom (DoF) altogether,  and the continuation of the SW 
  and breather 
  branches in Fig.~\ref{f1b}(c) needs about 1h on a laptop computer. 
  In contrast, the computation of all the steady branches in 
  Fig.\ref{f1}(c) and Fig.\ref{f1b}(c) only needs a few minutes. 
  Moreover, we checked that all the steady state results remain unchanged 
  for double the resolution, i.e., 200 spatial points, hence 1800 spatial DoF. 
  
  b) There are further 
  Hopf points on the spatially inhomogeneous branches, and the bifurcating 
  branches generally lead to ``breathing peaks''. However, none of these 
  appear to be stable, as discussed via DNS next, and therefore we restrict 
  to the green branch, which is also the most interesting as its bifurcation 
  leads to the gain of stability of the A branch.  
\end{remark}

\begin{figure}[ht]
  \includegraphics[width=0.9\textwidth]{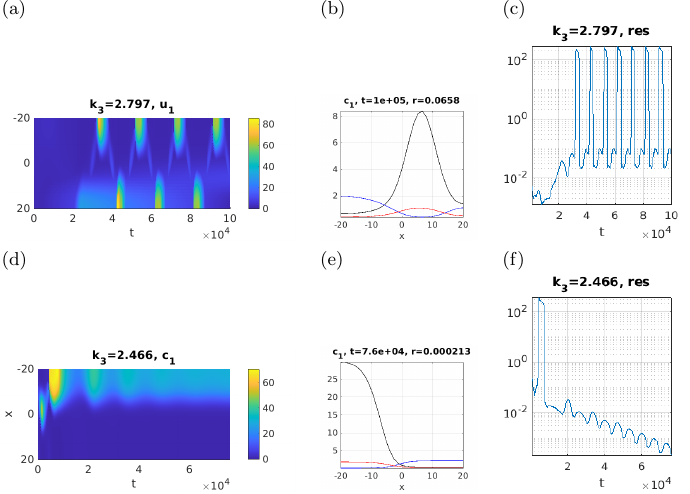}
  \vspace{-3mm}
\caption{{\small DNS for perturbations of SW$_2$ (a--c) and SW$_3$ (d--e).  
\label{f1d}}}
\end{figure}

The stability of SWs (of time periodic orbits in general) can be analyzed 
via their Floquet multipliers, which also determine the bifurcations {\em 
from} time periodic orbits such as period doubling or torus bifurcations, 
 \cite[\S3.4]{p2pbook}. However, the numerical 
computation of the Floquet multipliers 
is a difficult and expensive task and often prone to 
numerical instabilities, in particular for relaxation oscillation type 
of solutions as in Fig.\ref{f1b}. Therefore, to illustrate (in)stability of 
the SWs, in Fig.~\ref{f1d} we run time integration (aka direct numerical 
simulation DNS) from (perturbations of) 
selected SWs. 
That means we let $\tilde c(x)=c_{{\rm SW}}(0,x)$ (time $t=0$ slice) 
and add $0.1\cos(\pi x/20)$ to the first component of $\tilde c$ 
(which does not change the masses $m$), to obtain a perturbed 
initial condition for the DNS. 

Fig.\ref{f1d}(a-c) shows 
the result of doing so for SW$_2$ (see \cite{SI} for the DNS method 
used, i.e., linearly implicit Euler). (a) shows a space--time plot 
of $c_1$, (b) a snapshot of $c_1$ (and $c_2$ and $c_3$) at the stopping time, 
and (c) the ``residual'' $\|G(t)\|_\infty$, i.e., the 
amplitude of the dynamics. This again shows the relaxation oscillation 
behavior as $\|G(t)\|_\infty$ quickly shoots up when a peak develops 
at either the left or right boundary. In any case, 
from (a--c) we see that $c(t,\cdot)$ goes back to SW$_2$, indicating 
stability of SW$_2$, although this is just {\em one} initial 
perturbation. 
(d--f) show the results for perturbing SW$_3$. Here the solution 
converges to the steady front A (where some influence of 
damped complex eigenvalues can be seen in the oscillatory behavior 
of $\|G(t)\|_\infty$ in (e)), yielding instability of SW$_3$. 
Additionally we remark that also SW$_1$ is unstable, as expected 
from the subcritical bifurcation of the SW branch, as small perturbations 
yield convergence to $\oc$ here (not shown). 
Thus, in summary we find that we have a rather small 
range $k_3\in(k_{3a},k_{3b})$ where the SWs are (possibly) 
stable, $2.47<k_{3a}<k_{3b}<3$. The breathers are unstable, as they 
bifurcate subcritically at the gain of stability of the A branch.

In Fig.\ref{f1c} we consider the ``natural parametrization'' (NP) setting 
of \reff{rd2}, i.e., again using \reff{bp1} we now 
combine \reff{rd2} with \reff{q2} initialized 
at $k_3=5.5$, instead of \reff{q1} in Fig.\ref{f1}. Thus, while 
in Fig.~\ref{f1} $m\in\R^3$ changes with $k_3$ as determined 
by \reff{eq:mono_para}, in Fig.\ref{f1c}  all 
branches (and solutions) shown are at fixed $m\approx (12.4, 25, 640)$. 
The differences are not significant at large $k_3$ ($k_3>3$, say) 
but interesting changes happen at smaller $k_3$: 
\begin{compactitem}
\item The transcritical intersection of $\oc$ and $\wt{\oc}$ at $k_3\approx 
0.8$ in Fig.\ref{f1} now becomes imperfect. That means that the homogeneous 
branch $\hat{\oc}$ initially approximately follows the $\oc$ branch, 
but then continuously transitions (as one branch, without bifurcation) 
to approximately the former $\tilde{\oc}$ branch.  
Such breakups of transcritical bifurcations into imperfect bifurcations 
(separate branches which only ``almost intersect'') 
are well known in normal--form theory of bifurcations, see, e.g., 
\cite[\S2.5]{p2pbook}, and the references therein. Note that the black 
homogeneous branch $\hat{\oc}$ in Fig.\ref{f2b} only shows one half 
of the recombination under NP 
of the $\oc$ and $\wt{\oc}$ branches from MP. 
\item At small $k_3$, the $\hat{\oc}$ branch continues to grow 
(in $\|c_1\|_2$), 
 while the MP $\wt{\oc}$ decreased. 
\end{compactitem}
\begin{figure}[h]
  \includegraphics[width=0.9\textwidth]{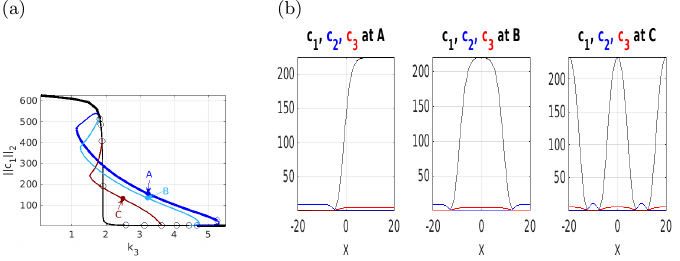}
  \vspace{-3mm}
\caption{{\small The NP setting, $\xi=(2,1,1)$ fixed, $k_9=0.5$, 
continuation in $k_3$, constraints \reff{q2} with $m$ initialized at $k_3=5.5$. 
Breakup of the $\oc$ --- $\wt{\oc}$  bifurcation from Fig.\ref{f1} 
to an imperfect bifurcation. 
\label{f1c}}}
\end{figure}

Similar transitions to imperfection also occur 
in other parameter regimes (e.g., $k_9=1$ as in Fig.\ref{f1b}) when 
going from MP to NP. However, in summary we find both effects of 
going from MP to NP 
{\em not particularly significant} 
for the problem from a pattern formation perspective; --they mainly 
illustrate that the BDs naturally change with the choice of constraints. 
On the other hand, the occurrence of Hopf bifurcations when going 
from $k_9=0.5$ in Fig.\ref{f1} to $k_9=1$ in Fig.\ref{f1b} and Fig.\ref{f1d} is 
significant, and in \S\ref{2dsec} we further explore this in 2D.

\subsection{2D}\label{2dsec}
Fig.\ref{f2a} shows a basic 
BD for \reff{rd1} on $D_{R}$=disk of radius $R$, here $R=10$, 
in the NP setting, using again \reff{bp1} but with \reff{bp1b}, i.e., 
$k_9=1$, to also discuss  Hopf bifurcations again and in particular 
to compute RWs. 
The black branch $\oc$ of homogeneous steady states is hence the analog of the 
black branch from Fig.\ref{f1c}, and the primary patterned steady 
branches belong to Bessel functions 
$v_{mn}$ with $(m,n)=(1,1)$ for  {\tt A}, $(m,n)=(1,2)$ for {\tt B},  
and $(m,n)=(0,1)$ for {\tt C} (Mexican hat). In particular, the {\tt A} 
and {\tt B} branches bifurcate as (subcritical) pitchforks, while 
the {\tt C} branch bifurcates transcritically. All of these 
reconnect to $\oc$ at small $k_3$ and large $\|c_1\|_2$, and 
$\oc$ regains stability at the reconnection of the {\tt A} branch. 

\begin{figure}[h]
  \includegraphics[width=0.9\textwidth]{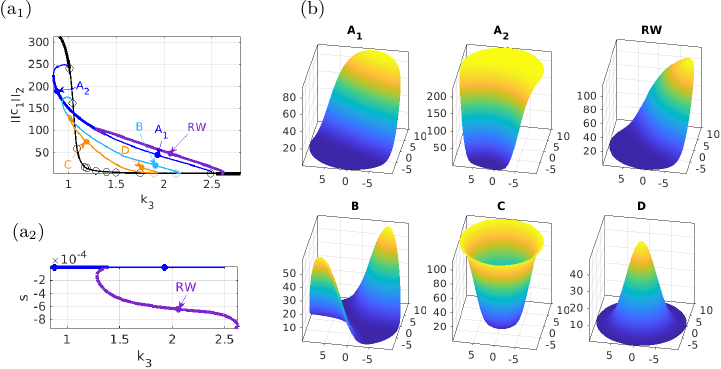}
  \vspace{-2mm}
\caption{{\small \reff{rd1},\reff{q2} over $D_{10}$, base parameters \reff{bp1} 
with $k_9=1$. (a) BD of basic steady states 
(A, dark blue; B, light blue; and C,D, orange), and one RW branch (violet);  
$\|c_1\|_2$ over $k_3$ in (a$_1$) and behavior of speed $s$ along RW branch 
in (a$_2$), showing a drift bifurcation of RW from A at $k_3\approx 1.4$.  
Samples in (b). 
\label{f2a}}}
\end{figure} 

Like in Fig.\ref{f1b}, the loss of stability of $\oc$ 
is  due to a Hopf bifurcation with 
a pair of complex conjugate eigenvalues crossing,  
with rather small 
imaginary parts $\im\mu_{1,2}(k_{3,c})\approx \pm 0.0008$. 
By the $O(2)$ equivariance of \reff{rd2} on disks, cf.Remark \ref{pcrem}b), 
we then have three bifurcating branches of POs: 
RW$_{\pm}$, and SWs. Since the numerical continuation of SWs is very 
expensive and since from experience  we do not expect SWs to be 
stable except possibly in narrow parameter regimes (cf.Remark \ref{numrem},   
and Fig.\ref{f1b}), here we focus on RWs, which can cheaply be computed as 
relative equilibria by going in a co--rotating frame with speed $s$, 
cf.\reff{rd2} and the PC \reff{cs}, with initial 
rotation speed $|s|=|\im\mu_{1}(k_{3,c})|$ at bifurcation. 
Thus continuing in $(k_3,s)$ we obtain Fig.\ref{f2a}(a$_2$) (for a clockwise 
RW, i.e., negative $s$), showing that $|s|$ decreases along the {\tt RW}  
branch, which reconnects to the {\tt A} branch with $s=0$ around 
$k_3=1.4$. Conversely, the bifurcation of {\tt RW} from {\tt A} 
is called a {\em drift bifurcation}, i.e., a steady 
bifurcation (real eigenvalues crossing the imaginary axis) leading 
to a (symmetry induced) drifting with speed increasing away from onset.  

\begin{figure}[h]
  \includegraphics[width=0.9\textwidth]{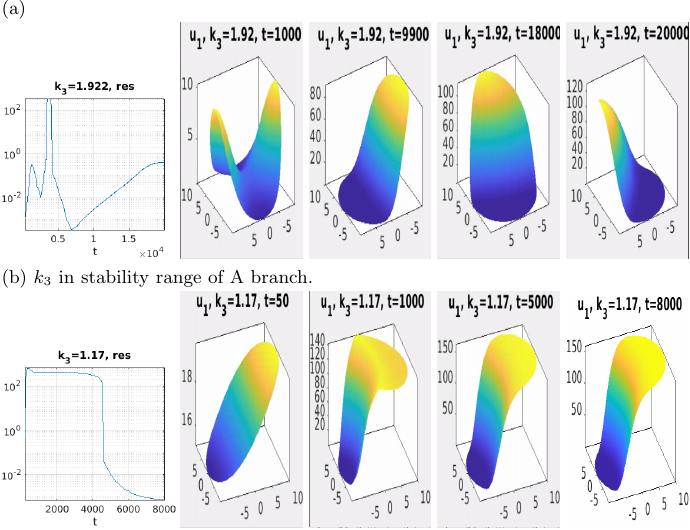}
  \vspace{-3mm}
\caption{{\small DNS for \reff{rd1} over $D_{10}$ with ICs as 
perturbations of $\oc$ at given $k_3$. (a) $k\approx 1.1922$; 
initial approach to A solution, then convergence to RW. 
(b) $k\approx 1.17$;  convergence to A.  
\label{f2b}}}
\end{figure}

Importantly, the RWs are linearly (i.e., spectrally) 
stable between the right and 
left fold of the RW branch, while the A branch only gains stability 
at the drift bifurcation of the RWs. This is confirmed by DNS in 
Fig.\ref{f2b}, also illustrating 
rather large domains of attraction of the respective stable solutions. 
In (a), with $k_3=1.922$, we choose  
the initial condition $c|_{t=0}=\oc$ with 
$0.05\cos(2\pi x/10)+C_1$ added to the first component, where 
$C_1$ is chosen such that $m_3$  does not change (and hence all of $m$, cause 
$m_1,m_2$ are not affected anyway). This yields an initial approach 
to the (unstable) A branch, but then emergence of a rotating waves. 
This can best be seen by the fact that 
the residual goes to a nonzero constant. 
In (b), at $k_1=1.17$, we choose 
$c|_{t=0}=\oc$ with $2\sin(\pi x/20)$ (to shorten transients) added 
to the first component, 
and obtain convergence to the A branch. 

\section{Discussion}\label{dsec} 

Traditionally, deriving instability conditions for
networks with a large number of species is mathematically
challenging. By restricting our focus to systems that admit a
monomial parametrization of their positive steady states, the search
for instabilities by analyzing the signs of the leading and constant
coefficients of the linearized system's characteristic polynomial
becomes feasible. This approach yields sufficient conditions for
instability that take the form of simple polynomial inequalities in
the reaction rate constants and diffusion coefficients. Furthermore,
we translate these algebraic conditions into conditions on the spatial 
domain size $|\Omega|$ for one-, two-, and three-dimensional spaces.

We apply this framework to a nine-species network modeling the
sequential and distributive double (de-)phosphorylation of a protein,  
and show that if the catalytic constants and the diffusion
coefficients of the enzyme-substrate complexes satisfy condition
(\ref{eq:kc_Km_d_ineq}), then the homogeneous steady state is
destabilized, resulting in the emergence of spatial patterns on
appropriately sized domains. This result highlights the practical
utility of our method in extracting conditions from complex
biochemical models.

To validate and extend these analytical predictions, we perform
numerical continuation and direct numerical simulation using
the \pdep\ package. The numerical analysis confirms our theoretical
predictions. On a 1D domain, the primary loss of stability is
confirmed to be of a steady long-wave type, giving rise to
subcritical pitchfork bifurcations. However, the numerical
exploration also reveals complex dynamics beyond the scope of our
determinant-based analysis. By adjusting the rate constant $k_9$, the
primary instability transitions to a Hopf bifurcation. In 1D, this
yields time-periodic standing waves and breathers exhibiting
relaxation oscillations, while on a 2D disk, it leads to drift
bifurcations and the emergence of stable rotating waves.

While our approach provides a starting point for finding sufficient
conditions for pattern formation in the network~(\ref{mapk}), it has
inherent limitations. The instabilities identified analytically are
only Turing-{\em like}; they represent long-wave (or sideband)
instabilities where the first instability is always with respect to
nonhomogeneous solutions of the maximal wavelength allowed by the
finite domain. The existence of genuine Turing instabilities --
characterized by an intrinsic critical wave number independent of the
domain size -- remains an open problem for this specific PDE system,
despite extensive numerical searches across different parameter
combinations. Furthermore, because our analytical method relies on
the product of eigenvalues, it cannot detect the Hopf bifurcations
observed numerically. Thus, while our analytical conditions serve as a 
guide for identifying steady-state pattern formation, future
theoretical work will be required to capture the full spectrum of
spatio-temporal dynamics of the system~(\ref{mapk}), including genuine
Turing instabilities.

\bibliographystyle{amsplain}
\bibliography{turing_crn}

\end{document}